\documentclass{article}
\usepackage{amsmath,amsfonts,epsfig}
\usepackage{amsmath,amsfonts,amssymb}
\usepackage{booktabs}
\usepackage{bm}
\usepackage{amsmath,amsfonts,epsfig}
\usepackage{graphics}
\usepackage{supertabular}
\usepackage{amsmath,amsfonts,amssymb}
\numberwithin{equation}{section}
\usepackage{amsthm}
\usepackage{mathtools}
\usepackage{amsfonts}
\usepackage{cite}
\usepackage{graphicx}
\usepackage{epstopdf}
\usepackage[numbers]{natbib}
\usepackage{float,epsfig, floatflt,here}
\usepackage{color}
\usepackage[colorlinks=true,citecolor=blue,linkcolor=blue]{hyperref}
\usepackage{fancyhdr}
\usepackage{etoolbox}
\usepackage{ dsfont }
\usepackage{accents}
\usepackage{bm}

\newtheorem{definition}{Definition}[section]

\newcommand{\norm}[1]{\| #1 \|}
\newcommand{\tnorm}[1]{{\left\vert\kern-0.25ex\left\vert\kern-0.25ex\left\vert #1\right\vert\kern-0.25ex\right\vert\kern-0.25ex\right\vert}}
\newcommand{\jump}[1]{\left[ #1 \right]}

\newcommand{\vertiii}[1]{{\left\vert\kern-0.25ex\left\vert\kern-0.25ex\left\vert #1
		\right\vert\kern-0.25ex\right\vert\kern-0.25ex\right\vert}}
\newcommand{\avg}[1]{\{\{ #1 \}\}_{e}}

\newtheorem{thm}{Theorem}[section]

\newtheorem{lemma}{Lemma}[section]

\newtheorem{rem}{Remark}[section]

\newtheorem{cor}{Corollary}[section]

\newtheorem{exm}{Example}[section]

\numberwithin{figure}{section}
\numberwithin{table}{section}
\usepackage{subcaption}
\title{\normalsize Numerical analysis of the second-order time-dependent saddle point Maxwell system via a parameter-free discontinuous Galerkin method: The first optimal ${\bf L}^{2}$-norm error estimates}
\date{}
\author{{\normalsize Achyuta Ranjan Dutta Mohapatra}\thanks{Department of Mathematics,
		Indian Institute of Technology Guwahati, North Guwahati- 781039, India (\texttt{\tt  achyutar@iitg.ac.in}).}\and
	{\normalsize Bhupen Deka }\thanks{Corresponding author, Department of Mathematics,
		Indian Institute of Technology Guwahati, North Guwahati- 781039, India ({\tt bdeka@iitg.ac.in}).}}

\date{}
\begin{document}
	\maketitle
	\begin{abstract}
		We present a novel parameter-free discontinuous Galerkin (dG) finite element method (FEM) for the time-dependent Maxwell system formulated as a saddle point problem. 
		We establish the stability of the proposed semi-discrete problem and derive optimal error estimates in energy and \( {\bf L}^{2} \) norms for the electric field variable, as well as in \( L^{2} \) norm for the potential function. To the best of our knowledge, this work provides the first optimal \( {\bf L}^{2} \)-norm error analysis for the second-order time-dependent saddle point Maxwell equations using any variants of FEMs. Additionally, we propose several complete discrete time-integrators and verify the optimal convergence results through examples in both 2D and 3D setups.
	\end{abstract}	
	\vspace{.01in}
	
	
	\section{\normalsize Introduction}\label{sec1}
	The electromagnetic phenomena are primarily governed by the Maxwell equations, which are a system of first-order curl-div systems (cf. \cite{monk1992analysis,chen2000finite}) posed in a domain $\Omega\times (0, T)$ and given as follows:
	\begin{eqnarray}\label{model1}\left\{
		\begin{array}{ll}
			\alpha^{-1}{\bf H}_{t}=-\nabla\times{\bf E},\\
			\nabla\times{\bf H}=\epsilon{\bf E}_{t}-{\bf J},\\
			\nabla\cdot(\epsilon {\bf E}) = \rho,\\
			\nabla\cdot(\alpha {\bf H}) = 0.
		\end{array}
		\right.\end{eqnarray}
	The corresponding variables \({\bf E}\), \({\bf H}\), \({\bf J}\), and $\rho$ represent the electric field, magnetic field, vector current density, and scalar charge density functions, respectively. Further \(\alpha\) and \(\epsilon\) relate to the inverse of magnetic permeability and electric permittivity, respectively, and are assumed to be constants throughout the domain. 	
	Further \(\Omega \subset \mathbb{R}^{d}\, (d=2,3)\) represents a bounded convex polytopal domain with its boundary denoted as $\partial\Omega$. The notation $T$ has the notion of final time.
	
	Upon elimination of the magnetic field variable ${\bf H}$ from the Maxwell system \eqref{model1}, one gets the following second-order space-time Maxwell system posed in the domain $\Omega\times(0, T)$ supplemented with appropriate initial and boundary conditions given as
	\begin{eqnarray}\label{model2}
		\left\{
		\begin{array}{ll}
			\epsilon{\bf E}_{tt}+\nabla\times (\alpha\nabla\times {\bf E})={\bf F}\; \mbox{in}~\Omega\times (0,T), \\
			\nabla\cdot(\epsilon {\bf E}) = \rho\; \mbox{in}~\Omega\times (0,T), \\
			{\bf E}({\bf x},0)={\bf u}({\bf x}),\,\,
			{\bf E}_t({\bf x},0)={\bf v}({\bf x})\;\;\mbox{on}\; \Omega,\\
			\boldsymbol{\eta}\times{\bf E}=0\;\;\mbox{on} ~\partial \Omega\times (0,T).
		\end{array}
		\right.
	\end{eqnarray}
	
	In this context, \(\boldsymbol{\eta}\) denotes the outward unit normal vector to the boundary \(\partial \Omega\). The functions \(\mathbf{u} : \Omega \rightarrow \mathbb{R}^{d}\) and \(\mathbf{v} : \Omega \rightarrow \mathbb{R}^{d}\) represent the initial functions, while \({\bf F} := {\bf J}_t \in [L^{2}(\Omega)]^{d}\).
	
	To relax the divergence condition mentioned in the second equation of \eqref{model2}, we follow the approach outlined in \cite{chen2000finite}. We introduce a Lagrange multiplier \(P\), which transforms the problem into a saddle point system expressed as follows:
	\begin{eqnarray}\label{model3}
		\left\{
		\begin{array}{ll}
			\epsilon{\bf E}_{tt}+\nabla\times (\alpha\nabla\times {\bf E})-\epsilon\nabla P={\bf F}\; \mbox{in}~\Omega\times (0,T), \\
			\nabla\cdot(\epsilon {\bf E}) = \rho\; \mbox{in}~\Omega\times (0,T), \\
			{\bf E}({\bf x},0)={\bf u}({\bf x}),\,\,
			{\bf E}_t({\bf x},0)={\bf v}({\bf x})\;\;\mbox{on}\; \Omega,\\
			\boldsymbol{\eta}\times{\bf E}=0,\;P=0\;\mbox{on} ~\partial \Omega\times (0,T).
		\end{array}
		\right.
	\end{eqnarray}
	
	There is a considerable literature based on the numerical approximation for the first-order time-dependent Maxwell equations (see  \cite{zhao2004analysis,makridakis1995time,monk1991mixed,li2015priori,da2022virtual,cockburn2004locally} and the references therein). However, for the second-order Maxwell problem, only a few articles deal with its numerical approximations.
	In \cite{monk1992analysis},	Monk proposed a classical FEM for the second-order time dependent Maxwell equations and deduced semi-discrete error estimates in the energy norm while using curl-conforming N{\'e}d{\'e}lec edge elements of type 1 and type 2 and ${\bf H}^{1}(\Omega)$-conforming finite elements. However, for obtaining optimal ${\bf L}^{2}$-norm estimates, curl-conforming N{\'e}d{\'e}lec edge elements were necessary. Ciarlet et al. (cf. \cite{ciarlet1999fully}) presented the complete discrete optimal convergence results for the second-order Maxwell equations using a backward Euler scheme in time and a curl-conforming FEM with the linear  N{\'e}d{\'e}lec edge elements of type 2 for spatial discretizations. Radu and Egger (cf. \cite{egger2021second}) proposed a second-order curl-conforming explicit FEM with mass lumping for the Maxwell equations, where they devised a second-order inexact Galerkin approximation of Maxwell's
	equations with a block diagonal mass matrix achieving the same accuracy and flexibility
	of standard finite element approximations. Further, the optimal convergence of the errors was shown in the ${\bf H}(\mathrm{curl})$-norms for the semi and complete discrete explicit schemes. Recently, in \cite{radu2025yee}, the authors presented a Yee-like FEM for the Maxwell equations and demonstrated optimal convergence analysis for the error in an energy norm. The interior penalty discontinuous Galerkin (IPDG) framework for the semi-discrete formulation of the second-order time-dependent Maxwell system was first analyzed in detail by Grote et al. in \cite{grote2007interior}, where they established optimal convergence rates in the energy norm subject to appropriate regularity of the exact solution. Their subsequent work (see \cite{grote2008interior}) extended this analysis to obtain optimal error bounds in the ${\bf L}^{2}$ norm. Further, Mitkova et al. (cf. \cite{grote2010explicit}) investigated explicit local time-stepping techniques combined with spatial IPDG discretizations for Maxwell's equations in both conducting and non-conducting media. More recently, Shields et al. \cite{shields2017weak} proposed an implicit weak Galerkin finite element approach for Maxwell's equations and proved optimal convergence properties with respect to a discrete energy norm. In \cite{duttamohapatra2025}, we had extended an explicit skeletal dG scheme for the second-order Maxwell equations and deduced optimal energy norm estimates for the error.
	The above-stated methods for second-order time-dependent Maxwell's equations implicitly assume that the electric field satisfies the Gauss law (see the third equation of \eqref{model1}).
	
	As a saddle point hyperbolic system, performing error analysis for the Maxwell equations \eqref{model3} can be quite complex, which is why there are relatively few publications addressing the numerical analysis of model \eqref{model3}.
	Our literature survey indicates that the first study on the time-dependent saddle point Maxwell problem was conducted by Assous et al. \cite{assous1993finite}. In their work, they utilized a modified Taylor-Hood element to approximate the electric field and potential functions. Subsequently, the authors in \cite{chen2000finite} proposed a conforming finite element method for the spatial discretization of the hyperbolic Maxwell saddle point interface problem. They achieved optimal convergence results for the electric field in the ${\bf H}(\mathrm{curl})$-norm, employing a linear curl-conforming element for the electric field approximation and continuous piecewise linear elements for the potential variable, while utilizing a backward Euler scheme for the temporal discretization. However, they did not provide error estimates for the potential function.
	More recently, Qi et al. \cite{qi2025decoupled} presented an optimal convergence analysis of the Maxwell system \eqref{model3}. They derived estimates in an “energy” semi-norm for the error in the electric field and the error of the potential function in the \(L^{2}\)-norm, using weak Galerkin finite element methods (WG-FEMs) for spatial discretization.
	
	The traditional primal dG schemes generally incorporate a non-physical penalty parameter to maintain well-posedness. However, some recent families of FEMs, such as hybridized high-order (HHO) methods (cf. \cite{di2014arbitrary}) and weak Galerkin methods (cf. \cite{wang2013weak}), do not require such parameters. The primary objective of this article is to develop suitable numerical solutions for the second-order time-dependent saddle point Maxwell system \eqref{model3} using parameter-free dG methods for spatial discretization.
	
	The key highlights of this work are as follows:
	
	\begin{itemize}
		\item We present a spatial discretization method that utilizes a parameter-free discontinuous Galerkin numerical scheme for solving the time-dependent saddle point Maxwell equations.
		
		\item  We conduct a stability analysis of the associated continuous-in-time dG scheme and establish optimal convergence of the errors in both the electric field and potential functions under suitable norms by using an appropriately defined Ritz (elliptic) projection. 
		
		\item Under lower regularity assumptions ${\bf E}(t)\in {\bf H}^{1+l}(\Omega)$ $(l>0)$ for the electric field, we have established $\mathcal{O}(h^{1+l})$ convergence results in the ${\bf L}^{2}$-norm for the electric field error (see Theorem \ref{optL2low}).

		\item  Additionally, we propose time-integration schemes based on backward Euler and Crank-Nicolson type discretizations for the discrete-in-time model. We extensively discuss various numerical experiments, which cover both two-dimensional and three-dimensional test cases.
	\end{itemize}

	To the best of our knowledge, this article is the first to provide error estimates in the ${\bf L}^{2}$-norm for the second-order time-dependent Maxwell saddle point system.
	
	The structure of this article is as follows. In Section \ref{sec2}, the Maxwell equations are introduced, and their corresponding variational problem is derived. Some basic notations are also discussed. Section \ref{sec3} introduces some fundamentals required for presenting the parameter-free dG method. The semi-discretization of the model problem is done in Section \ref{sec4} using the parameter-free dG technique, followed by detailed derivations of stability and optimal error estimates of the corresponding problem. In Section \ref{sec6} we present some numerical computations in both 2D/3D domains to verify the theoretical results. Finally Section \ref{sec6} concludes the article.

	\section{\normalsize Basic notations and weak formulation}	\label{sec2}
	
	The standard notations for Sobolev spaces and norms have been followed in this paper. Consider a domain $N\subseteq\Omega\subset \mathbb{R}^d\; (d=2, 3)$. The Hilbert Sobolev space of scalar-valued functions is denoted as 
	$H^{l}(N)$, where $l\ge 0$ is an integer, equipped with the inner product $(\cdot, \cdot)_{l, N}$, 
	semi-norm $|\cdot |_{l, N}$ and norm $\norm{\cdot}_{l,N}$. For $N=\Omega$, we drop the subscript from the norm and inner product. Again, for vector-valued Sobolev Spaces, we represent them as ${\bf H}^l(N):=\left[H^{l}(N)\right]^d$.
	The notion of curl for a vector function ${\bf V}=(V_{1},V_{2})^{t}$ and a scalar function $V$ in two dimensions are given in the following way
	\begin{eqnarray*}
		\nabla\times{\bf V}=\frac{\partial V_{2}}{\partial x}-\frac{\partial V_{1}}{\partial y},\;
		\nabla\times V= \left(\frac{\partial V}{\partial y},-\frac{\partial V}{\partial x}\right)^{t}.
	\end{eqnarray*}
	Motivated from the definitions of curl for scalar and vector functions in two dimensional setup, we now introduce the following space
	$${\bf H}^l(\mbox{{curl}};N):=\{{\bf v}\in {\bf H}^l(N):\; \nabla\times{\bf v}\in [{H}^{l}(N)]^{2d-3}\}.$$
	Denote by ${\bf H}^0(\mbox{{curl}};N)={\bf H}(\mbox{{curl}};N)$, which is equipped with the graph norm
	$$\norm{\textbf{v}}_{\text{curl},N}^{2}:=\norm{\textbf{v}}_{N}^2+\norm{\nabla\times\textbf{v}}_{N}^2.$$
	Further, a subspace of ${\bf H}(\mathrm{curl};N)$ with zero tangential trace is introduced and is given as
	$${\bf H}_{0}(\mbox{{curl}};N):=\{{\bf v}\in {\bf H}(\mathrm{curl};N):\; \boldsymbol{\eta}\times{\bf v} = 0 \mbox{ on } \partial N \}.$$ 
	
	For an integer $p$ with $1\le p \le \infty$, we also define the standard Bochner spaces ${\bf L}^p(0,T;{\cal B})$, where
	$\cal B$ is a real Banach space with norm $\|.\|_{\cal B}$, consisting of all measurable functions $\Phi:J\to {\cal B}$ for which
	\begin{eqnarray*}
		\|\Phi\|_{{\bf L}^p(0,T; {\cal B})}&=& \Big(\int_0^T\|\Phi(t)\|_{\cal B}^pdt\Big)^{\frac{1}{p}}< \infty\;\;\;\mbox{for}\;\;1\le p < \infty,\\
		\|\Phi\|_{{\bf L}^{\infty}(0,T; {\cal B})}&=& \mbox{ess}\sup_{t\in [0,T]}\|\Phi(t)\|_{\cal B}< \infty\;\;\;\mbox{for}\;\; p=\infty.
	\end{eqnarray*}
	We denote by ${\bf H}^{l}(0,T;{\cal B})$, $l$ is an integer such that $1\le l <\infty$, the space of all measurable functions $\Phi:(0,T)\to {\cal B}$ for which
	\begin{eqnarray*}
		\|\Phi\|_{{\bf H}^{l}(0,T;{\cal B})}=\Bigg(\sum_{j=0}^l\int_0^T\bigg{\|}
		\frac{\partial^j \Phi (t)}{\partial t^j}\bigg{\|}^2_{\cal
			B}dt\Bigg)^{\frac{1}{2}}< \infty.
	\end{eqnarray*}
	When no risk of confusion exists, we shall write ${\bf L}^{2}({\cal B})$ for ${\bf L}^{2}(0,T;{\cal B})$, ${\bf L}^{\infty}({\cal B})$ for ${\bf L}^{\infty}(0,T;{\cal B})$ and ${\bf H}^{l}({\cal B})$ for ${\bf H}^{l}(0,T;{\cal B})$.
	
	We have used the notation $C$ throughout this article to denote a positive constant whose value changes according to the context and is independent of space and time mesh sizes but can depend on the final time $T$.
	
	The weak formulation of \eqref{model3} seeks $({\bf E}(t),P(t))\in {\bf H}_0(\mathrm{curl};\Omega)\times H^{1}_{0}(\Omega)$, $t\in (0,T)$, such that ${\bf E}(\cdot,0)={\bf u}$, ${\bf E}_t(\cdot,0)={\bf v}$ and satisfies the following equations:
	\begin{eqnarray*}
		\left\{
		\begin{array}{l}
			(\epsilon{\bf E}_{tt},\textbf{v}) + (\alpha\nabla\times{\bf E},\nabla\times\textbf{v}) - (\epsilon{\bf v},\nabla P) = ({\bf F}, \textbf{v}) \, \forall \textbf{v} \in {\bf H}_0(\mathrm{curl};\Omega),\\
			(\epsilon{\bf E},\nabla q)=-(\rho,q)\,  \forall q\in H^{1}_{0}(\Omega).
		\end{array}
		\right.
	\end{eqnarray*}
	
	\section{\normalsize The discontinuous Galerkin discretizations}\label{sec3}
	
	The spatial discretizations of the domain $\Omega$ are motivated from the work of Wang et al. \cite{wang2014weak}. Let $\mathcal{K}_{h}$ be a quasi-uniform polygonal/polyhedral partition of the domain $\Omega$ in 2D/3D. Assume that the set of all edges/faces are denoted by $\mathcal{F}_h$ and further  $\mathcal{F}_h^{0}$ represents the set of all interior edges/faces. For any element $K$ of the mesh partition $\mathcal{K}_{h}$, let $h_{K}$ represent its diameter and thus the mesh size is given as $h:=\max_{K\in\mathcal{K}_h}h_{K}$.
	
	We now introduce the notions of tangential and normal jumps alongside the averages of scalar and vector valued functions. Suppose $e$ is an intra-element edge/face shared by two arbitrary elements (say) $K_1,\,K_2\in \mathcal{K}_h$ if $\textbf{n}_1$ and $\textbf{n}_2$ represent the unit outward normal vectors on $e$ for $K_1$ and $K_2$, respectively, then the normal and tangential jumps are defined, respectively, as
	\begin{eqnarray*}
		\jump{{\bf v}}_{N,e}=\textbf{n}_1\cdot{\bf v}_1+\textbf{n}_2\cdot{\bf v}_2,\; \jump{{\bf v}}_{T,e}=\textbf{n}_1\times{\bf v}_1+\textbf{n}_2\times{\bf v}_2.
	\end{eqnarray*}
	Further, the average of ${\bf v}$ is defined as
	\begin{eqnarray*}
		\avg{{\bf v}}=\frac{{\bf v}_{1}+{\bf v}_{2}}{2}.
	\end{eqnarray*}
	The average and jump for a scalar-valued function ${v}$ are defined as:
	\begin{eqnarray*}
		\avg{v}=\frac{v_{1}+v_{2}}{2},\; \jump{v}_e=\textbf{n}_1{v}_1+\textbf{n}_2{v}_2.
	\end{eqnarray*}
	Here, ${\bf v}_{i}={\bf v}|_{K_{i}}$, where $i=1,2$, and the same holds for $v_{i}$.
	If $e$ is a boundary edge, i.e. $e\subset\partial\Omega$, the tangential and normal jumps of ${\bf v}$ along edge $e$ are defined as 
	\begin{eqnarray*}
		\jump{\textbf{v}}_{N,e}=\boldsymbol{\eta}\cdot\textbf{v}|_{K},\, \jump{\textbf{v}}_{T,e}=\boldsymbol{\eta}\times\textbf{v}|_{K}.
	\end{eqnarray*}
	Again, the average of ${\bf v}$ on any $e\subset\partial\Omega$ is defined as
	\begin{eqnarray*}
		\avg{{\bf v}}={\bf v}|_{K}.
	\end{eqnarray*}
	Similarly, the average and jump for a scalar-valued function ${v}$ on any $e\subset\partial\Omega$ are defined as:
	\begin{eqnarray*}
		\avg{v}=v|_{K},\; \jump{v}_e=\boldsymbol{\eta}{v}|_K.
	\end{eqnarray*}
	The notations $\textbf{v}|_K$ and $v|_K$ stands for values of $\textbf{v}$ and $v$ on $K$, respectively.
	
	The dG finite element solution space for approximating the electric field ${\bf E}$ is defined as
	\begin{equation*}
		{\bf V}_h:=\left\{{\bf v}_{h}\in {\bf L}^{2}(\Omega): {\bf v}|_{K}\in {\bf P}_{k}(K),\,\text{for all}\, K\in\mathcal{K}_h\right\}\nonumber.
	\end{equation*}
	Again, the dG finite element solution space for approximating the potential $P$ is given by
	\begin{equation*}
		Q_{h}=\left\{q_{h}\in L^{2}(\Omega):q_{h}|_{K}\in P_{k-1}(K),\,\text{for all}\, K\in\mathcal{K}_h\right\}.
	\end{equation*}
	Here, ${\bf P}_{k}(K):=[P_{k}(K)]^{d}$ and $P_{k}(K)$ is the set of polynomials defined on the element $K$ having degree less or equal to $k$, for $k\ge 1$.
	
	Further, we introduce the following subspaces
	\begin{eqnarray*}
		{\bf V}_h^{0}&=&\{{\bf v}_{h}\in{\bf V}_{h}: \boldsymbol{\eta}\times{\bf v}_{h}|_e=0, \text{ for any } e\subset\partial\Omega\},\\
		Q_{h}^{0}&=&\left\{q_{h}\in Q_{h}:q_{h}|_{e}=0 \text{ for any } e\subset\partial\Omega\right\}.
	\end{eqnarray*}
	\begin{definition}[cf. Definition 2.3., \cite{tang2023modified}]
		The discrete modified weak curl of any ${\bf v}_{h}\in {\bf V}_{h}$ is defined as a unique computable quantity $\nabla_{w}\times{\bf v}_{h}\in[P_{k-1}(K)]^{2d-3}$ satisfying the following equation,
		\begin{equation}\label{modweakcurl}
			(\nabla_{w}\times {\bf v}_h,w)_{K}:=({\bf v}_{h},\nabla\times w)_K +\langle{\bf n}\times \avg{{\bf v}_h},w\rangle_{\partial K},\;\forall w\in [P_{k-1}(K)]^{2d-3}.
		\end{equation}
	\end{definition}
	\begin{definition}[cf. Definition 1.1. \cite{wang2014modified}]
		The discrete modified weak gradient of any $q_{h}\in Q_{h}$ is defined as a unique computable quantity $\nabla_{w}q_{h}\in {\bf P}_{k}(K)$ satisfying the following equation,
		\begin{equation}\label{modweakgrad}
			(\nabla_{w} q_{h},{\bf r})_{K}:=-(q_{h},\nabla\cdot {\bf r})_K +\langle \avg{q_{h}},{\bf n}\cdot {\bf r}\rangle_{\partial K}\;\forall {\bf r}\in {\bf P}_{k}(K).
		\end{equation}
	\end{definition}
	In the definitions \eqref{modweakcurl} and \eqref{modweakgrad}, the notations $(\cdot,\cdot)_{K}$ and $\langle\cdot,\cdot\rangle_{\partial K}$ denotes the standard $L^{2}$-inner products on $K$ and $\partial K$, respectively.
	
	Now, we introduce some standard ${\bf L}^{2}$-projection operators
	\begin{itemize}
		\item ${\bf T}_h: [L^{2}(\Omega)]^{d}\rightarrow{\bf V}_{h}^{0}$,
		\item $\mathcal{T}_h:[L^{2}(K)]^{2d-3}\rightarrow[P_{k-1}(K)]^{2d-3}$,
		\item $L_h: [L^{2}(\Omega)]^{d}\rightarrow Q_h^{0}$.
	\end{itemize}
	We recall some useful inequalities from \cite{wang2014weak},
	\begin{itemize}
		\item If $u\in H^1(K)$, then we have the following trace inequality
		\begin{equation}\label{trace}
			\norm{u}_e^2\leq C\left(h_K^{-1}\norm{u}_K^2+h_K\norm{\nabla u}_K^2\right),
		\end{equation}
		where $K\in\mathcal{K}_{h}$ and $e$ represents the edge/face of an element $K$.
		\item Further, for each piece-wise polynomial $\phi$ of degree $n$ on $\mathcal{K}_h$, we have the following inverse inequality
		\begin{equation}\label{inverse}
			\norm{\nabla\phi}_{K}\le C(n)h^{-1}\norm{\phi}_{K}.
		\end{equation}
	\end{itemize}
	Before proceeding to describe the parameter-free dG method, we introduce some bilinear forms $A({\cdot,\cdot}):{\bf V}_{h}\times{\bf V}_{h}\to \mathbb{R}$, $\mathcal{S}({\cdot,\cdot}):{\bf V}_{h}\times{\bf V}_{h}\to \mathbb{R}$ and $B(\cdot,\cdot):{\bf V}_{h}\times Q_{h} \to\mathbb{R}$ given by
	\begin{eqnarray*}
		A({\bf v}_h,{\bf w}_h)&:=&\sum_{K\in\mathcal{K}_h}(\alpha\nabla_{w}\times{\bf v}_h,\nabla_{w}\times{\bf w}_{h})_{K}+\mathcal{S}({\bf v}_h,{\bf w}_{h}),\\
		\mathcal{S}({\bf v}_h,{\bf w}_{h})&:=&\sum_{K\in\mathcal{K}_{h}}h^{-1}_{K}(\langle\jump{{\bf v}_h}_{T,e},\jump{{\bf w}_h}_{T,e}\rangle_{\partial K}+\langle\jump{{\bf v}_h}_{N,e},\jump{{\bf w}_h}_{N,e}\rangle_{\partial K\setminus\partial\Omega}),~~~~~~~~\\
		B({\bf v}_h,q_{h})&:=&\sum_{K\in\mathcal{K}_h}(\epsilon{\bf v}_h,\nabla_{w}q_{h})_{K}.
	\end{eqnarray*}
	We define some semi-norms  $\tnorm{\cdot}$ and $\tnorm{\cdot}_{0}$ on the function spaces and ${\bf V}_{h}$ and $Q_{h}$, respectively, given as
	\begin{eqnarray*}
		\tnorm{{\bf v}_{h}}^{2}&:=&A({\bf v}_{h},{\bf v}_{h}),\\
		\tnorm{q_{h}}_{0}^{2}&:=&\sum_{K\in\mathcal{K}_{h}}\norm{\nabla_{w} q_{h}}_{K}^{2}+\sum_{e\in\mathcal{F}_{h}}h^{-1}\norm{\jump{q_{h}}_{e}}_{e}^{2}.
	\end{eqnarray*}
	In fact it is easy to see that the semi-norm $\tnorm{\cdot}_{0}$ is indeed a norm in the space $Q_{h}^{0}$.
	Again, we introduce a norm on the space ${\bf V}_{h}^{0}$, given by
	\begin{eqnarray*}
		\tnorm{{\bf v}_{h}}_{1}^{2}&:=&\norm{{\bf v}_{h}}^{2}+\tnorm{{\bf v}_{h}}^{2}
		.
	\end{eqnarray*}
	\begin{lemma}[cf. Lemma 4.3., \cite{duttamohapatra2025}]\label{energyinverse}
		The following inverse inequality in discrete norm $\tnorm{\cdot}_{0}$ holds true for any 
		$q_{h}\in Q_{h}^{0}$
		\begin{equation*}
			\tnorm{q_{h}}_{0}\le Ch^{-1}\norm{q_{h}}.
		\end{equation*}
	\end{lemma}
	\begin{lemma}[cf. Lemma 4.5., \cite{duttamohapatra2025}]\label{infsup}
		For any $q_{h}\in Q_h^{0}$, there exists a ${\bf w}_{h}\in {\bf V}_{h}^{0}$ such that following holds
		\begin{eqnarray*}
			B({\bf w}_{h},q_{h})&=&\norm{q_{h}}^{2},\\
			\tnorm{{\bf w}_{h}}_{1}&\le&C\norm{q_{h}}.
		\end{eqnarray*}
	\end{lemma}
	Consider the arbitrary pair $(\boldsymbol{\phi}_h,\psi_{h})\in {\bf V}_{h}^{0}\times Q_{h}^{0}$, then the Ritz projection pair $(\mathbb{R}_h,\mathbb{E}_{h})$ is given in the following way: The map  $\mathbb{R}_h: {\bf H}^{1+k}(\Omega)\cap {\bf H}_{0}(\mathrm{curl};\Omega) \to {\bf V}_h^0$ and $\mathbb{E}_{h}:H^{k}(\Omega)\cap H^{1}_{0}(\Omega)\to Q_{h}^{0}$ is defined as the unique solution of the following discrete variational problem 
	\begin{eqnarray}\label{Ritz}
		\left\{
		\begin{array}{ll}
			A(\mathbb{R}_h {\bf z}, \boldsymbol{\phi}_h) - B(\boldsymbol{\phi}_{h},\mathbb{E}_{h}q) = ({\bf f}_{{\bf z},q}, \boldsymbol{\phi}_h),\\
			B(\mathbb{R}_h {\bf z},\psi_{h})=-(g_{{\bf z},q},\psi_{h}).
		\end{array}
		\right.
	\end{eqnarray}
	where ${\bf z}\in {\bf H}^{1+k}(\Omega)\cap {\bf H}_{0}(\mathrm{curl};\Omega)$, $q\in H^{k}(\Omega)\cap H^{1}_{0}(\Omega)$, ${\bf f}_{{\bf z},q}=\nabla\times(\alpha\nabla\times {\bf z})-\epsilon\nabla q$ and $g_{{\bf z},q}=\nabla\cdot(\epsilon {\bf z})$ in $\Omega$. Hence, the pair $(\mathbb{R}_h {\bf v},\mathbb{E}_{h}q)$ can be realized as numerical approximation obtained using the parameter-free dG method (cf. \cite{duttamohapatra2025}) of the following elliptic problem: Find $({\bf z},q)\in ({\bf H}^{1+k}(\Omega)\cap {\bf H}_{0}(\mathrm{curl};\Omega))\times (H^{k}(\Omega)\cap H^{1}_{0}(\Omega))$ satisfying 
	\begin{eqnarray}\label{dual}
		\left\{
		\begin{array}{ll}
			\nabla\times (\alpha\nabla\times{\bf z})-\epsilon\nabla q={\bf f}_{{\bf z},q}\; \mbox{in}~\Omega, \\
			\nabla\cdot(\epsilon {\bf z}) = g_{{\bf z},q}\; \mbox{in}~\Omega. \\
		\end{array}
		\right.
	\end{eqnarray}
	\begin{lemma}[cf. Theorem 4.16, Theorem 4.18, Corollary 4.20, \cite{duttamohapatra2025}]\label{ritzestimate}
		Let the exact solution of the elliptic problem \eqref{dual} has the regularity $({\bf z},q)\in ({\bf H}^{1+k}(\Omega)\cap {\bf H}_{0}(\mathrm{curl};\Omega))\times (H^{k}(\Omega)\cap H^{1}_{0}(\Omega))$, then there exists a positive constant $C$ independent of mesh size $h$ satisfying:
		\begin{eqnarray}
			\left\{
			\begin{array}{ll}
				\tnorm{{\bf T}_{h}{\bf z}-\mathbb{R}_{h}{\bf z}}_{1}\le Ch^{k}\left(\norm{{\bf z}}_{1+k}+\norm{q}_{k}\right),\nonumber\\
				\norm{{\bf T}_{h}{\bf z}-\mathbb{R}_{h}{\bf z}}\le Ch^{\delta+k}\left(\norm{{\bf z}}_{1+k}+\norm{q}_{k}\right),\nonumber\\
				\norm{L_{h}q-\mathbb{E}_{h}q}\le Ch^{k}\left(\norm{{\bf z}}_{1+k}+\norm{q}_{k}\right).
			\end{array}		\right.
		\end{eqnarray}
		Here $0<\delta\le 1$ is the Sobolev regularity exponent of the solution of corresponding dual problem of \eqref{dual} (see \cite{duttamohapatra2025} for more details). 
	\end{lemma} 
	\section{\normalsize The continuous time parameter-free dG algorithm}\label{sec4}
	
	In this section, we propose the semi-discrete parameter-free dG method for the Maxwell system \eqref{model3} and derive the optimal error estimates under appropriate norms for the electric field and potential function.
	
	\noindent\rule{\textwidth}{0.4pt}
	\noindent{\bf The semi-discrete parameter-free dG algorithm 1:}
	For $t>0,$  find $({\bf E}_{h}(t),P_{h}(t))\in{\bf V}_{h}^{0}\times Q_{h}^{0}$ satisfying
	\begin{eqnarray}
		(\epsilon{\bf E}_{htt},{\bf v}_h)+A({\bf E}_{h},{\bf v}_h)-B({\bf v}_{h},P_h)&=&({\bf F},{\bf v}_h)\;\forall{\bf v}_h\in{\bf V}_h^0,\label{sdalg1}\\
		B({\bf E}_{h},q_h)&=&-(\rho,q_{h})\; \forall q_h\in Q_{h}^{0},\label{sdalg2}\\
		{\bf E}_h(\textbf{x},0)&=&\mathbb{R}_h{\bf u}(\textbf{x}),\,\textbf{x}\in\Omega,\label{sdalg3}\\
		{\bf E}_{ht}(\textbf{x},0)&=&\mathbb{R}_h{\bf v}(\textbf{x}),\,\textbf{x}\in\Omega.\label{sdalg4}
	\end{eqnarray}	
	\noindent\rule{\textwidth}{0.4pt}
	\begin{rem}\label{rem2}
		The use of the modified weak curl and modified weak gradient operators in the algorithm \eqref{sdalg1}-\eqref{sdalg4} is the key to remove parameter dependent stabilization terms in the bilinear forms appearing the continuous-in-time dG algorithm.
	\end{rem}
	\begin{rem}\label{initialchoice}
		We assume that $P_{h}(0)=\mathbb{E}_{h}P(0)$. Setting $t\to 0^{+}$ in \eqref{sdalg1}, for any ${\bf v}_{h}\in {\bf V}_{h}^{0}$ we achieve
		\begin{eqnarray}\label{comp1}
			A({\bf E}_{h}(0),{\bf v}_h)-B({\bf v}_{h},P_h(0))&=&({\bf F}(0)-\epsilon{\bf E}_{htt}(0),{\bf v}_h).
		\end{eqnarray}
		Again, selecting $({\bf z},q)=({\bf E}(0),P(0))$ in \eqref{Ritz}, for any ${\bf v}_{h}\in {\bf V}_{h}^{0}$ we obtain
		\begin{eqnarray}\label{comp2}
			A(\mathbb{R}_h {\bf E}(0), {\bf v}_h) - B({\bf v}_h,\mathbb{E}_{h}P(0)) &=& ({\bf f}_{{\bf E}(0),P(0)})\nonumber\\&=&({\bf F}(0)-\epsilon{\bf E}_{tt}(0),{\bf v}_h)\nonumber\\&=&({\bf F}(0)-\epsilon{\bf T}_{h}{\bf E}_{tt}(0),{\bf v}_h).
		\end{eqnarray} 
		Comparing \eqref{comp1}-\eqref{comp2}, we get
		\begin{equation}\label{initialsecond}
			{\bf E}_{htt}(0)={\bf T}_{h}{\bf E}_{tt}(0).
		\end{equation}
	\end{rem}
	\begin{rem}\label{Pt}
		By selecting $({\bf z},q)=({\bf E}(t),P(t))$ and  differentiating the second equation of the problem \eqref{dual} thrice with respect to $t$ and taking the limit $t\to 0^{+}$ on both sides of it, we get $\nabla\cdot(\epsilon{\bf E}_{ttt}(0))=\rho_{ttt}(0)$. Next differentiating first equation of \eqref{dual} with respect to $t$, then taking the limit $t\to 0^{+}$ on both sides of it and finally apply the divergence operator on both sides, we get $-\nabla\cdot(\epsilon P_{t}(0))=\nabla\cdot{\bf F}_{t}(0)-\rho_{ttt}(0)$. Then due to $a$ $priori$ estimate of the elliptic problem 
		\begin{eqnarray*}
			\left\{
			\begin{array}{ll}
				-\nabla\cdot(\epsilon P_{t}(0))=\nabla\cdot{\bf F}_{t}(0)-\rho_{ttt}(0)\,\,\mathrm{ in }\,\,\Omega,\\
				P_{t}(0)=0\,\,\mathrm{on}\,\,\partial\Omega,
			\end{array}
			\right.
		\end{eqnarray*}
		we have
		\begin{eqnarray}\label{apr1}
			\norm{P_{t}(0)}_{2}\le C\left(\norm{{\bf F}_{t}(0)}_{1}+\norm{\rho_{ttt}(0)}\right).
		\end{eqnarray}
		Arguing in the same manner, for the elliptic problem
		\begin{eqnarray}\label{elliptcproblem}
			\left\{
			\begin{array}{ll}
				-\nabla\cdot(\epsilon P(0))=\nabla\cdot{\bf F}(0)-\rho_{tt}(0)\,\,\mathrm{ in }\,\,\Omega,\\
				P(0)=0\,\,\mathrm{on}\,\,\partial\Omega,
			\end{array}
			\right.
		\end{eqnarray}
		we have
		\begin{eqnarray}\label{apr2}
			\norm{P(0)}_{2}\le C\left(\norm{{\bf F}(0)}_{1}+\norm{\rho_{tt}(0)}\right).
		\end{eqnarray}
		Such estimates will be helpful in later part of the analysis.
	\end{rem}
	\begin{lemma}\label{bounded}
		The discrete bilinear form $A(\cdot,\cdot):{\bf V}_{h}\times{\bf V}_{h}\to \mathbb{R}$ satisfies the following continuity (boundedness) property i.e. for any ${\bf u}_{h}, {\bf v}_{h}\in {\bf V}_{h}$, we have
		\begin{eqnarray*}
			|A({\bf u}_{h},{\bf v}_{h})|\le C\tnorm{{\bf u}_{h}}\tnorm{{\bf v}_{h}}.
		\end{eqnarray*}
	\end{lemma}
	\begin{proof}
		The proof follows from a simple consequence of Cauchy-Schwarz inequality.
	\end{proof}
	\begin{lemma}[cf. Lemma 6.1., \cite{mu2015weak}]\label{approximation}
		Assume that $\mathcal{K}_h$ is a shape regular partition of the domain $\Omega$. Then for any $\;{\bf w}\in {\bf H}^{1+s}(\Omega)$ and $w\in H^{s}(\Omega)$, we have 
		\begin{eqnarray*}
			&&\sum_{K\in\mathcal{K}_h}h_{K}^{2m}\norm{{\bf w}-{\bf T}_{h}{\bf w}}_{m,K}^{2}\le C h^{2(1+s)}\norm{{\bf w}}_{1+s, \Omega}^{2},
			\\
			&&\sum_{K\in\mathcal{K}_h}h_{K}^{2m}\norm{\nabla\times{\bf w}-\mathcal{T}_h(\nabla\times{\bf w})}_{m,K}^{2}\le C h^{2s}\norm{{\bf w}}_{1+s, \Omega}^{2},\\
			&&\sum_{K\in\mathcal{K}_h}h_{K}^{2m}\norm{w-L_{h}w}_{m,K}^{2}\le Ch^{2s}\norm{w}_{s,\Omega}^{2}.
		\end{eqnarray*}
		where $m\in[0,1]$ and $s\in[0,k]$.
	\end{lemma}
	\begin{lemma}[cf. Lemma 4.7., \cite{mohapatra2025new}]\label{Techres2lem}
		Let ${\bf v}\in {\bf H}(\mathrm{curl};\Omega)$ be an arbitrary element and for any $\xi\in [P_{k-1}(K)]^{2d-3}$, following identity holds true
		\begin{equation*}
			(\alpha\nabla_{w}\times ({\bf T}_h{\bf v}),\xi)_K=(\mathcal{T}_h(\alpha\nabla\times{\bf v}),\xi)_K+\langle {\bf n}\times(\avg{{\bf T}_h{\bf v}}-{\bf v}),\alpha\xi\rangle_{\partial K}\, \forall K\in\mathcal{K}_h.
		\end{equation*}
	\end{lemma}
	Further, we present some different variants of the continuous Gronwall's inequality which will be used depending on the context.
	
	\begin{lemma}[Gronwall's Inequality 1, cf. Lemma 3.1, \cite{riviere2008discontinuous}]\label{g1}
		Let $f$, $g$, and $h$ be piecewise continuous nonnegative functions defined on an interval $(a,b)$.
		Assume that $g$ is nondecreasing. Assume that there exists a positive constant $C$,
		independent of $t$, such that
		\begin{equation*}
			\forall t \in (a,b), \qquad
			f(t) + h(t) \le g(t) + C \int_a^t f(s)\, ds .
		\end{equation*}
		Then,
		\begin{equation*}
			\forall t \in (a,b), \qquad
			f(t) + h(t) \le e^{C(t-a)}\, g(t).
		\end{equation*}
	\end{lemma}
	\begin{lemma}[Gronwall's Inequality 2, cf. Theorem 1, \cite{sever2003some}]
		Let $x$, $\Psi$ and $\chi$ be real continuous functions defined on $[a,b]$, with
		$\chi(t) \ge 0$ for $t \in [a,b]$. Suppose that on $[a,b]$ we have the inequality
		\begin{equation*}
			x(t) \le \Psi(t) + \int_a^t \chi(s)\, x(s)\, ds .
		\end{equation*}
		Then
		\begin{equation*}
			x(t) \le \Psi(t) + \int_a^t \chi(s)\, \Psi(s)
			\exp\!\left( \int_s^t \chi(u)\, du \right) ds ,
		\end{equation*}
		for all $t \in [a,b]$.
	\end{lemma}
	Upon selecting a positive constant $K$ in place of $\chi(t)$ i.e. $\chi(t)=K$ in the above Gronwall's inequality, we thus state a modified Gronwall's result below:
	\begin{lemma}[Modified Gronwall's Inequality]\label{g2}
		Let $x$ and $\Psi$ be real continuous functions defined on $[a,b]$ with
		$\Psi(t) \ge 0$ for $t \in [a,b]$, with $a>0$. Suppose that on $[a,b]$ we have the inequality
		\begin{equation*}
			x(t) \le \Psi(t) + K\int_a^t \, x(s)\, ds .
		\end{equation*}
		Then
		\begin{equation*}
			x(t) \le \Psi(t) + K\exp(Kt)\int_a^t \Psi(s)\,ds,
		\end{equation*}
		for all $t \in [a,b]$.
	\end{lemma}
	Now, we present the stability of the numerical scheme \eqref{sdalg1}-\eqref{sdalg4} for the case when $\epsilon=1$. For any other $\epsilon>0$, the proof is  easily extendable. 
	\begin{lemma}
		Assume that the supplied data of model \eqref{model3} satisfy
		\begin{eqnarray*}
			&&{\bf u}\in{\bf H}^{2}(\Omega)\cap{\bf H}_{0}(\mathrm{curl};\Omega), 
			\,
			{\bf v}\in{\bf H}^{2}(\Omega)\cap{\bf H}_{0}(\mathrm{curl};\Omega),\\
			&&{\bf F}\in{\bf H}^{2}\!\left(0,T;{\bf H}^{2}(\Omega)\right), 
			\, 
			\rho\in H^{4}\!\left(0,T;L^{2}(\Omega)\right).
		\end{eqnarray*}
		Then the semi-discrete parameter-free dG scheme \eqref{sdalg1}--\eqref{sdalg4}
		satisfies the following stability estimates:
		\begin{eqnarray}\label{stab1}
			\norm{{\bf E}_{ht}(t)}^{2}
			+\tnorm{{\bf E}_{h}(t)}^{2}
			\le
			C\left(
			\norm{{\bf u}}_{2}^{2}
			+\norm{{\bf v}}_{2}^{2}
			+\norm{{\bf F}}_{{\bf H}^{2}({\bf H}^{2}(\Omega))}^{2}
			+\norm{\rho}_{H^{4}(L^{2}(\Omega))}^{2}
			\right),~~
		\end{eqnarray}
		and
		\begin{eqnarray}\label{stab2}
			\norm{P_{h}(t)}^{2}
			\le
			C\left(
			\norm{{\bf u}}_{2}^{2}
			+\norm{{\bf v}}_{2}^{2}
			+\norm{{\bf F}}_{{\bf H}^{2}({\bf H}^{2}(\Omega))}^{2}
			+\norm{\rho}_{H^{4}(L^{2}(\Omega))}^{2}
			\right).
		\end{eqnarray}
	\end{lemma}
	\begin{proof}
		Differentiating \eqref{sdalg2} with respect to $t$, we have
		\begin{equation}\label{st1}
			B({\bf E}_{ht},q_h)=-(\rho_{t},q_{h})\,\forall q_{h}\in Q_{h}^{0}.
		\end{equation}
		Substituting $q_{h}=P_{h}$ in \eqref{st1}, then adding it with \eqref{sdalg1} by replacing ${\bf v}_{h}={\bf E}_{ht}$ and applying some standard inequalities, we achieve
		\begin{eqnarray*}
			\frac{1}{2}\frac{d}{dt}\left(\norm{{\bf E}_{ht}}^{2}+\tnorm{{\bf E}_{h}}^{2}\right)&=&({\bf F},{\bf E}_{ht})-(\rho_{t}, P_{h})\nonumber\\
			&\le&\norm{{\bf F}}\norm{{\bf E}_{ht}}+\norm{\rho_{t}}\norm{ P_{h}}\nonumber\\
			&\le&\frac{1}{2}\norm{{\bf F}}^{2}+\frac{1}{2}\norm{{\bf E}_{ht}}^{2}+\frac{1}{2}\norm{\rho_{t}}^{2}+\frac{1}{2}\norm{ P_{h}}^{2}.
		\end{eqnarray*}
		Integrating both sides of above inequality from $0$ to $t$ $(0<t\le T)$, we derive
		\begin{eqnarray*}
			\norm{{\bf E}_{ht}(t)}^{2}+\tnorm{{\bf E}_{h}(t)}^{2}
			&\le&\norm{{\bf E}_{ht}(0)}^{2}+\tnorm{{\bf E}_{h}(0)}^{2}\nonumber\\
			&&+\int_{0}^{t}\norm{{\bf F}}^{2}ds+\int_{0}^{t}\norm{{\bf E}_{hs}}^{2}ds\nonumber\\
			&&+\int_{0}^{t}\norm{\rho_{s}}^{2}ds+\int_{0}^{t}\norm{ P_{h}}^{2}ds.
		\end{eqnarray*}
		From Gronwall's inequality 1 (cf. Lemma \ref{g1}), it follows that
		\begin{eqnarray}\label{st2}
			\norm{{\bf E}_{ht}(t)}^{2}+\tnorm{{\bf E}_{h}(t)}^{2}
			&\le&C\norm{{\bf E}_{ht}(0)}^{2}+C\tnorm{{\bf E}_{h}(0)}^{2}+C\int_{0}^{t}\norm{{\bf F}}^{2}ds\nonumber\\
			&&+C\int_{0}^{t}\norm{\rho_{s}}^{2}ds+C\int_{0}^{t}\norm{ P_{h}}^{2}ds.~~~~~~~~
		\end{eqnarray}
		Now we need to find a bound the term $\int_{0}^{t}\norm{ P_{h}}^{2}ds$. Note that if $\norm{P_{h}}=0$ then our proof is complete and we shall have a stability estimate for the discrete electric field. Hence we proceed by assuming  $\norm{P_{h}}\ne0$.
		
		We differentiate \eqref{sdalg1} and \eqref{st1} with respect to $t$ to have
		\begin{eqnarray*}
			\left\{
			\begin{array}{ll}
				({\bf E}_{httt},{\bf v}_h)+A({\bf E}_{ht},{\bf v}_h)-B({\bf v}_{h},P_{ht})=({\bf F}_{t},{\bf v}_h)\;\forall{\bf v}_h\in{\bf V}_h^0,\nonumber\\
				B({\bf E}_{htt},q_h)=-(\rho_{tt},q_{h})\,\forall q_{h}\in Q_{h}^{0}.
			\end{array}
			\right.
		\end{eqnarray*}
		Replacing ${\bf v}_{h}={\bf E}_{htt}$ and $q_{h}=P_{ht}$ in the above equations and adding them, we achieve 
		\begin{eqnarray*}
			\frac{1}{2}\frac{d}{dt}\left(\norm{{\bf E}_{htt}}^{2}+\tnorm{{\bf E}_{ht}}^{2}\right)=({\bf F}_{t},{\bf E}_{htt})-(\rho_{tt}, P_{ht}).
		\end{eqnarray*}
		Integrating both sides of above equation from $0$ to $t$ $(0<t\le T)$, further using the fact ${\bf E}_{htt}(0)={\bf T}_{h}{\bf E}_{tt}(0)$ (cf. Remark \ref{initialchoice}), integration by parts, Cauchy-Schwarz and Young's inequality, we deduce
		\begin{eqnarray*}
			\norm{{\bf E}_{htt}(t)}^{2}+\tnorm{{\bf E}_{ht}(t)}^{2}&=&\norm{{\bf E}_{htt}(0)}^{2}+\tnorm{{\bf E}_{ht}(0)}^{2}\nonumber\\
			&&+2\int_{0}^{t}({\bf F}_{s},{\bf E}_{hss})ds-2\int_{0}^{t}(\rho_{ss}, P_{hs})ds\nonumber\\
			&=&\norm{{\bf T}_{h}{\bf E}_{tt}(0)}^{2}+\tnorm{{\bf E}_{ht}(0)}^{2}+2\int_{0}^{t}({\bf F}_{s},{\bf E}_{hss})ds\nonumber\\
			&&-2(\rho_{tt}(t),P_{h}(t))+2(\rho_{tt}(0),P_{h}(0))\nonumber\\
			&&+2\int_{0}^{t}(\rho_{sss}, P_{h})ds\nonumber\\
			&\le&\norm{{\bf E}_{tt}(0)}^{2}+\tnorm{{\bf E}_{ht}(0)}_{1}^{2}+\int_{0}^{t}\norm{{\bf F}_{s}}^{2}ds\nonumber\\
			&&+\int_{0}^{t}\norm{{\bf E}_{hss}}^{2}ds+2\norm{\rho_{tt}(t)}\norm{P_{h}(t)}+\norm{\rho_{tt}(0)}^{2}\nonumber\\
			&&+\norm{P_{h}(0)}^{2}+\int_{0}^{t}\norm{\rho_{sss}}^{2}ds+\int_{0}^{t}\norm{P_{h}}^{2}ds.				
		\end{eqnarray*}
		By Modified Gronwall's inequality (cf. Lemma \ref{g2}), we deduce
		\begin{eqnarray}\label{st3}
			\norm{{\bf E}_{htt}(t)}^{2}&\le&C\norm{{\bf E}_{tt}(0)}^{2}+C\tnorm{{\bf E}_{ht}(0)}_{1}^{2}+C\norm{P_{h}(0)}^{2}+2\norm{\rho_{tt}(t)}\norm{P_{h}(t)}\nonumber\\
			&&+C\norm{\rho_{tt}(0)}^{2}+C\int_{0}^{t}\norm{{\bf F}_{s}}^{2}ds+C\int_{0}^{t}\norm{\rho_{sss}}^{2}ds\nonumber\\
			&&+C\int_{0}^{t}\norm{P_{h}}^{2}ds+C\int_{0}^{t}\norm{\rho_{ss}(s)}\norm{P_{h}(s)}ds\nonumber\\
			&\le&C\norm{{\bf E}_{tt}(0)}^{2}+C\tnorm{{\bf E}_{ht}(0)}_{1}^{2}+C\norm{P_{h}(0)}^{2}+2\norm{\rho_{tt}(t)}\norm{P_{h}(t)}\nonumber\\
			&&+C\norm{\rho_{tt}(0)}^{2}+C\int_{0}^{t}\norm{{\bf F}_{s}}^{2}ds+C\int_{0}^{t}\norm{\rho_{ss}}^{2}ds\nonumber\\
			&&+C\int_{0}^{t}\norm{\rho_{sss}}^{2}ds+C\int_{0}^{t}\norm{P_{h}}^{2}ds.~~~~~~~~
		\end{eqnarray}
		From inf-sup condition (cf. Lemma \ref{infsup}) for $P_{h}\in Q_{h}^{0}$ there exists some ${\bf v}_{h}\in {\bf V}_{h}^{0}$, such that $\norm{P_{h}}^{2}=B({\bf v}_{h},P_{h})$ and $\tnorm{{\bf v}_{h}}_{1}\le C\norm{P_{h}}$. Applying above stated arguments, then by Cauchy-Schwarz inequality and using Lemma \ref{bounded}, we can write
		\begin{eqnarray*}
			\norm{P_{h}}^{2}&=&B({\bf v}_{h},P_{h})\nonumber\\
			&=&({\bf E}_{htt},{\bf v}_h)+A({\bf E}_{h},{\bf v}_h)-({\bf F},{\bf v}_h)\nonumber\\
			&\le&\norm{{\bf E}_{htt}}\tnorm{{\bf v}_{h}}_{1}+C\tnorm{{\bf E}_{h}}\tnorm{{\bf v}_{h}}_{1}+\norm{{\bf F}}\tnorm{{\bf v}_{h}}_{1}\nonumber\\
			&\le&C\norm{{\bf E}_{htt}}\norm{P_{h}}+C\tnorm{{\bf E}_{h}}\norm{P_{h}}+C\norm{{\bf F}}\norm{P_{h}}.
		\end{eqnarray*}
		Hence, we deduce
		\begin{eqnarray}\label{st4}
			\norm{P_{h}}\le C\norm{{\bf E}_{htt}}+ C\tnorm{{\bf E}_{h}} + C\norm{{\bf F}}.
		\end{eqnarray}
		Combining \eqref{st2}-\eqref{st4} and applying Young's inequality, we have
		\begin{eqnarray*}
			\norm{P_{h}}^{2}
			&\le&C\norm{{\bf E}_{htt}}^{2}+C\tnorm{{\bf E}_{h}}^{2}+C\norm{{\bf F}}^{2}\nonumber\\
			&\le&C\norm{{\bf E}_{tt}(0)}^{2}+C\tnorm{{\bf E}_{ht}(0)}_{1}^{2}+C\norm{P_{h}(0)}^{2}+2\norm{\rho_{tt}(t)}\norm{P_{h}(t)}\nonumber\\
			&&+C\norm{\rho_{tt}(0)}^{2}+C\int_{0}^{t}\norm{{\bf F}_{s}}^{2}ds+C\int_{0}^{t}\norm{\rho_{ss}}^{2}ds+C\int_{0}^{t}\norm{\rho_{sss}}^{2}ds\nonumber\\
			&&+C\int_{0}^{t}\norm{P_{h}}^{2}ds+C\norm{{\bf E}_{ht}(0)}^{2}+C\tnorm{{\bf E}_{h}(0)}^{2}\nonumber\\
			&&+C\int_{0}^{t}\norm{{\bf F}}^{2}ds+C\int_{0}^{t}\norm{\rho_{s}}^{2}ds+C\int_{0}^{t}\norm{ P_{h}}^{2}ds+C\norm{{\bf F}}^{2}.
		\end{eqnarray*}
		Simplifying the above inequality, we obtain
		\begin{eqnarray*}
			\norm{P_{h}}^{2}
			&\le&C\norm{{\bf E}_{tt}(0)}^{2}+C\tnorm{{\bf E}_{ht}(0)}_{1}^{2}+C\tnorm{{\bf E}_{h}(0)}^{2}+C\norm{P_{h}(0)}^{2}\nonumber\\
			&&+C\max_{t\in[0,T]}\norm{{\bf F}(t)}^{2}+2\norm{\rho_{tt}(t)}^{2}+\frac{1}{2}\norm{P_{h}(t)}^{2}+C\norm{\rho_{tt}(0)}^{2}\nonumber\\
			&&+C\int_{0}^{t}\norm{\rho_{s}}^{2}ds+C\int_{0}^{t}\norm{\rho_{ss}}^{2}ds+C\int_{0}^{t}\norm{\rho_{sss}}^{2}ds\nonumber\\
			&&+C\int_{0}^{t}\norm{{\bf F}}^{2}ds+C\int_{0}^{t}\norm{{\bf F}_{s}}^{2}ds+C\int_{0}^{t}\norm{ P_{h}}^{2}ds.
		\end{eqnarray*}
		Again by Gronwall's inequality 1 (cf. Lemma \ref{g1}), we derive
		\begin{eqnarray}\label{st5}
			\norm{P_{h}}^{2}
			&\le&C\norm{{\bf E}_{tt}(0)}^{2}+C\tnorm{{\bf E}_{ht}(0)}_{1}^{2}+C\tnorm{{\bf E}_{h}(0)}^{2}\nonumber\\
			&&+C\norm{P_{h}(0)}^{2}+C\max_{t\in[0,T]}\norm{{\bf F}(t)}^{2}+C\max_{t\in[0,T]}\norm{\rho_{tt}(t)}^{2}\nonumber\\
			&&+C\int_{0}^{t}\norm{\rho_{s}}^{2}ds+C\int_{0}^{t}\norm{\rho_{ss}}^{2}ds+C\int_{0}^{t}\norm{\rho_{sss}}^{2}ds\nonumber\\
			&&+C\int_{0}^{t}\norm{{\bf F}}^{2}ds+C\int_{0}^{t}\norm{{\bf F}_{s}}^{2}ds.
		\end{eqnarray}
		It is easy to see from triangle inequality, Lemma \ref{ritzestimate} and \eqref{apr2} that
		\begin{eqnarray}\label{st6}
			\tnorm{{\bf E}_{h}(0)}^{2}&\le&	\tnorm{{\bf E}_{h}(0)}_{1}^{2}\nonumber\\
			&\le&C\tnorm{\mathbb{R}_{h}{\bf E}(0)-{\bf E}(0)}_{1}^{2}+C\tnorm{{\bf E}(0)}_{1}^{2}\nonumber\\
			&\le&C(\norm{{\bf u}}_{2}^{2}+\norm{P(0)}_{1}^{2})\nonumber\\
			&\le&C\left(\norm{{\bf u}}_{2}^{2}+\norm{{\bf F}(0)}_{1}^{2}+\norm{\rho_{tt}(0)^{2}}\right).
		\end{eqnarray}
		In the same manner, from \eqref{apr1} we have
		\begin{eqnarray}\label{st7}
			\tnorm{{\bf E}_{ht}(0)}_{1}^{2}
			&\le&C\tnorm{\mathbb{R}_{h}{\bf E}_{t}(0)-{\bf E}_{t}(0)}_{1}^{2}+C\tnorm{{\bf E}_{t}(0)}_{1}^{2}\nonumber\\
			&\le&C(\norm{{\bf v}}_{2}^{2}+\norm{P_{t}(0)}_{1}^{2})\nonumber\\
			&\le&C\left(\norm{{\bf v}}_{2}^{2}+\norm{{\bf F}_{t}(0)}_{1}^{2}+\norm{\rho_{ttt}(0)^{2}}\right),
		\end{eqnarray}
		and,
		\begin{eqnarray}\label{st8}
			\norm{P_{h}(0)}^{2}&\le&C\norm{\mathbb{E}_{h}P(0)-P(0)}^{2}+C\norm{P(0)}^{2}\nonumber\\
			&\le&C(\norm{{\bf u}}_{2}^{2}+\norm{P(0)}_{1}^{2})\nonumber\\
			&\le&C\left(\norm{{\bf u}}_{2}^{2}+\norm{{\bf F}(0)}_{1}^{2}+\norm{\rho_{tt}(0)^{2}}\right).
		\end{eqnarray}
		Further, a bound for $\norm{{\bf E}_{tt}(0)}$ is obtained by setting $t\to0^{+}$ in the first equation of model \eqref{model3} and later by using \eqref{apr2}, we derive
		\begin{eqnarray}\label{st9}
			\norm{{\bf E}_{tt}(0)}^{2}&=&\norm{{\bf F}(0)+\epsilon\nabla P(0)-\nabla\times(\alpha\nabla\times{\bf E}(0))}^{2}\nonumber\\
			&\le&C\norm{{\bf F}(0)}^{2}+C\norm{P(0)}_{2}^{2}+\norm{{\bf u}}_{2}^{2}\nonumber\\
			&\le&C\norm{{\bf F}(0)}_{1}^{2}+\norm{\rho_{tt}(0)}^{2}+\norm{{\bf u}}_{2}^{2}.
		\end{eqnarray}
		Plugging in the bounds \eqref{st6}-\eqref{st9} in \eqref{st5} we achieve the stability bound (cf. \eqref{stab2}) for the discrete potential approximation. Then using \eqref{st6}, \eqref{st7} and \eqref{stab2} in \eqref{st2} we obtain the desired stability estimate (cf.  \eqref{stab1}) for the discrete electric field.
		
	\end{proof}
	\subsection{\normalsize Error analysis of the semi-discrete scheme}	
	The exact errors arising due to numerical approximations by the semi-discrete parameter-free dG scheme can be split into approximation and projected errors, respectively, given as
	\begin{eqnarray*}
		{\bf E}-{\bf E}_{h}&=&{{\bf E}-{\bf T}_h{\bf E}}+{\bf T}_h{\bf E}-{\bf E}_h,\nonumber\\
		P-P_{h}&=&P-L_{h}P+L_{h}P-P_{h}.
	\end{eqnarray*}
	Estimates for approximation error follow from Lemma \ref{approximation}, hence our objective reduces to find bounds for the projected error due to electric field and the potential function which are denoted by ${\bf e}_{h}$ and $\tilde{e}_h$, respectively, and are given as
	\begin{eqnarray*}
		{\bf e}_{h}: = {\bf T}_h{\bf E}-{\bf E}_h,\quad\mathrm{and}\quad
		\tilde{e}_{h}: = L_h P-P_h.
	\end{eqnarray*}
	For the purpose of error analysis in discrete energy norm, let ${\bf E}$ be the exact solution of \eqref{model3}, we introduce some bilinear forms for each $(\textbf{v}_h,q_{h})\in{\bf V}_h\times Q_{h}$, described as follows:
	\begin{eqnarray}\label{bilinearforms}
		\left\{
		\begin{array}{ll}
			\mathcal{M}_1({\bf E},{\bf v}_{h}):=\sum_{K\in\mathcal{K}_{h}}\langle {\bf n}\times(\avg{{\bf T}_h{\bf E}}-{\bf E}),\alpha\nabla_{w}\times{\bf v}_h\rangle_{\partial K},\nonumber\\
			\mathcal{M}_2({\bf E},{\bf v}_h):=\sum_{K\in\mathcal{K}_{h}}\langle{\bf n}\times({\bf v}_h-\avg{{\bf v}_{h}}),\alpha\nabla\times{\bf E}-\mathcal{T}_h(\alpha\nabla\times{\bf E})\rangle_{\partial K},\nonumber\\
			\mathcal{M}_3(P,{\bf v}_{h}):=\sum_{K\in\mathcal{K}_{h}}\langle\epsilon\left(P-\avg{L_{h}P}\right),{\bf n}\cdot{\bf v}_{h}\rangle_{\partial K},\nonumber\\
			\mathcal{M}_4({\bf E},q_{h}):=\sum_{K\in\mathcal{K}_{h}}\langle{\bf n}\cdot\left({\bf E}-{\bf T}_{h}{\bf E}\right),q_{h}-\avg{q_{h}}\rangle_{\partial K}.\nonumber
		\end{array}
		\right.
	\end{eqnarray}
	\begin{lemma}\label{semiderrorequation}
		Let \(({\bf E},P)\) be the exact solution of model \eqref{model3}. Then the following error equation holds for arbitrary
		\(({\bf v}_{h},q_h)\in {\bf V}_{h}^{0}\times Q_{h}^{0}\):
		\begin{eqnarray}\label{sderrorequation}
			\left\{
			\begin{array}{ll}
				(\epsilon{\bf e}_{htt},{\bf v}_h)+A({\bf e}_h,{\bf v}_h)-B({\bf v}_{h},\tilde{e}_h)=\sum_{i=1}^{2}\mathcal{M}_i({\bf E},{\bf v}_{h})+\mathcal{M}_3(P,{\bf v}_{h})+\mathcal{S}({\bf T}_h{\bf E},{\bf v}_h).\\
				B({\bf e}_{h},q_{h}) = \mathcal{M}_4({\bf E},q_{h}).
			\end{array}
			\right.
		\end{eqnarray}
	\end{lemma}
	\begin{proof}
		We begin the proof by testing an arbitrary  ${\bf v}_{h}\in {\bf V}_{h}^{0}$ in the first equation of model \eqref{model3} and later using integration by parts to arrive at
		\begin{eqnarray}
			({\bf F},{\bf v}_h)&=&(\epsilon{\bf E}_{tt},{\bf v}_h)+(\nabla\times(\alpha\nabla\times{\bf E}),{\bf v}_h)-(\epsilon\nabla P,{\bf v}_h)\nonumber\\
			&=&(\epsilon{\bf E}_{tt},\textbf{v}_h)+\sum_{K\in\mathcal{K}_h}(\nabla\times(\alpha\nabla\times{\bf E}),\textbf{v}_h)_{K}-\sum_{K\in\mathcal{K}_h}(\epsilon\nabla P,{\bf v}_h)_{K}\nonumber\\
			&=&(\epsilon{\bf E}_{tt},\textbf{v}_h)+\sum_{K\in\mathcal{K}_h}(\alpha\nabla\times{\bf E},\nabla\times{\bf v}_h)_K\nonumber\\
			&&+\sum_{K\in\mathcal{K}_h}\langle\textbf{n}\times(\alpha\nabla\times{\bf E}),{\bf v}_h-\avg{{\bf v}_{h}}\rangle_{\partial K}-\sum_{K\in\mathcal{K}_h}(\epsilon\nabla P,{\bf v}_h)_{K}.\label{erreqn1}~~~~~~~~~~
		\end{eqnarray}
		Here in the last equality we have used the continuity of flux to establish the fact
		$$\sum_{K\in\mathcal{K}_h}\langle\textbf{n}\times(\alpha\nabla\times{\bf E}),\avg{{\bf v}_{h}}\rangle_{\partial K}=0.$$
		Utilizing the definition of discrete modified weak curl \eqref{modweakcurl}, integration by parts and the properties of ${\bf L}^{2}$-projection $\mathcal{T}_{h}$, we can derive
		\begin{eqnarray}
			(\mathcal{T}_h(\alpha\nabla\times{\bf E}),\nabla_{w}\times{\bf v}_h)_{K}&=&({\bf v}_h,\nabla\times\mathcal{T}_h(\alpha\nabla\times{\bf E}))_{K}\nonumber\\
			&&+\langle\textbf{n}\times\avg{{\bf v}_h},\mathcal{T}_h(\alpha\nabla\times{\bf E})\rangle_{\partial K}\nonumber\\
			&=&(\nabla\times{\bf v}_h,\mathcal{T}_h(\alpha\nabla\times{\bf E}))_{K}\nonumber\\
			&&-\langle\textbf{n}\times({\bf v}_h
			-\avg{{\bf v}_h}),\mathcal{T}_h(\alpha\nabla\times{\bf E})\rangle_{\partial K}\nonumber\\
			&=&(\nabla\times{\bf v}_h,\alpha\nabla\times{\bf E})_{K}\nonumber\\
			&&-\langle{\bf n}\times({\bf v}_h
			-\avg{{\bf v}_{h}}),\mathcal{T}_h(\alpha\nabla\times{\bf E})\rangle_{\partial K}.~~~~~~
			\label{erreqn2}
		\end{eqnarray}
		Again, by Lemma \ref{Techres2lem} and \eqref{erreqn2}, we can write
		\begin{eqnarray}
			(\nabla\times{\bf v}_h,\alpha\nabla\times{\bf E})_{K} &=&(\alpha\nabla_{w}\times ({\bf T}_h{\bf E}),\nabla_{w}\times{\bf v}_h)_{K}\nonumber\\
			&&-\langle {\bf n}\times(\avg{{\bf T}_h{\bf E}}-{\bf E}),\alpha\nabla_{w}\times{\bf v}_h\rangle_{\partial K}
			\nonumber\\
			&&+\langle{\bf n}\times({\bf v}_h-\avg{{\bf v}_{h}}),\mathcal{T}_h(\alpha\nabla\times{\bf E})\rangle_{\partial K}.\label{erreqn3}
		\end{eqnarray}
		Applying integration by parts, properties of $L^{2}$-projection $L_{h}$ and using definition of discrete modified weak gradient \eqref{modweakgrad}, we deduce
		\begin{eqnarray}
			(\epsilon\nabla P,{\bf v}_{h})_{K}&=&-(P,\nabla\cdot({\epsilon\bf v}_{h}))_{K}+\langle P,{\bf n}\cdot(\epsilon{\bf v}_{h})\rangle_{\partial K}\nonumber\\
			&=&-(L_{h}P,\nabla\cdot({\epsilon\bf v}_{h}))_{K}+\langle P,{\bf n}\cdot(\epsilon{\bf v}_{h})\rangle_{\partial K}\nonumber\\
			&=&(\epsilon\nabla_{w}L_{h}P,{\bf v}_{h})_{K}-\langle\avg{L_{h}P}- P,{\bf n}\cdot(\epsilon{\bf v}_{h})\rangle_{\partial K}.\label{erreqn4}
		\end{eqnarray}
		Assembling \eqref{erreqn3}-\eqref{erreqn4} in \eqref{erreqn1} we achieve
		\begin{eqnarray}
			({\bf F},{\bf v}_h)&=&(\epsilon{\bf T}_{h}{\bf E}_{tt},\textbf{v}_h)+\sum_{K\in\mathcal{K}_{h}}(\alpha\nabla_{w}\times ({\bf T}_h{\bf E}),\nabla_{w}\times{\bf v}_h)_{K}\nonumber\\
			&&-\sum_{K\in\mathcal{K}_{h}}\langle {\bf n}\times(\avg{{\bf T}_h{\bf E}}-{\bf E}),\alpha\nabla_{w}\times{\bf v}_h\rangle_{\partial K}
			\nonumber\\
			&&+\sum_{K\in\mathcal{K}_{h}}\langle{\bf n}\times({\bf v}_h-\avg{{\bf v}_{h}}),\mathcal{T}_h(\alpha\nabla\times{\bf E})-\alpha\nabla\times{\bf E}\rangle_{\partial K}\nonumber\\
			&&-\sum_{K\in\mathcal{K}_{h}}(\epsilon\nabla_{w}L_{h}P,{\bf v}_{h})_{K}\nonumber\\
			&&+\sum_{K\in\mathcal{K}_{h}}\langle\avg{L_{h}P}- P,{\bf n}\cdot(\epsilon{\bf v}_{h})\rangle_{\partial K}.\label{erreqn5}
		\end{eqnarray}
		Adding $\mathcal{S}({\bf T}_h{\bf E},{\bf v}_{h})$ on both sides of \eqref{erreqn5} and later subtracting \eqref{sdalg1} from \eqref{erreqn5}, we get the first semi-discrete error equation of \eqref{sderrorequation}.
		
		Now, we test the second equation of \eqref{model3} with \( q_{h} \in Q_{h}^{0} \). Further, we apply Green's formula and then use property of the \( \mathbf{L}^{2} \)-projection \( \mathbf{T}_{h} \) along with the definition of discrete modified weak gradient \eqref{modweakgrad}. This sequence of steps leads to
		\begin{eqnarray}
			(\rho,q_{h}) & = & (\nabla\cdot (\epsilon{\bf E}),q_{h})\nonumber\\
			&=&\sum_{K\in\mathcal{K}_{h}}(\nabla\cdot (\epsilon{\bf E}),q_{h})_{K}\nonumber\\
			&=&-\sum_{K\in\mathcal{K}_{h}}(\epsilon{\bf E},\nabla q_{h})_{K} + \sum_{K\in\mathcal{K}_{h}}\langle{\bf n}\cdot(\epsilon{\bf E}),q_{h}\rangle_{\partial K}\nonumber\\
			&=&-\sum_{K\in\mathcal{K}_{h}}(\epsilon{\bf T}_{h}{\bf E},\nabla q_{h})_{K} + \sum_{K\in\mathcal{K}_{h}}\langle{\bf n}\cdot(\epsilon{\bf E}),q_{h}\rangle_{\partial K}\nonumber\\
			&=&\sum_{K\in\mathcal{K}_{h}}(\epsilon\nabla\cdot({\bf T}_{h}{\bf E}),q_{h})_{K} + \sum_{K\in\mathcal{K}_{h}}\langle{\bf n}\cdot\left({\bf E}-{\bf T}_{h}{\bf E}\right),\epsilon q_{h}\rangle_{\partial K}\nonumber\\
			&=&-\sum_{K\in\mathcal{K}_{h}}(\epsilon{\bf T}_{h}{\bf E},\nabla_{w}q_{h})_{K} +\sum_{K\in\mathcal{K}_{h}}\langle{\bf n}\cdot{\bf T}_{h}{\bf E},\epsilon\avg{q_{h}}\rangle_{\partial K}\nonumber\\
			&&+ \sum_{K\in\mathcal{K}_{h}}\langle{\bf n}\cdot\left({\bf E}-{\bf T}_{h}{\bf E}\right),\epsilon q_{h}\rangle_{\partial K}\nonumber\\
			&=&-\sum_{K\in\mathcal{K}_{h}}(\epsilon{\bf T}_{h}{\bf E},\nabla_{w}q_{h})_{K} +\sum_{K\in\mathcal{K}_{h}}\langle{\bf n}\cdot\left({\bf T}_{h}{\bf E}-{\bf E}\right),\epsilon\avg{q_{h}}\rangle_{\partial K}\nonumber\\
			&&+ \sum_{K\in\mathcal{K}_{h}}\langle{\bf n}\cdot\left({\bf E}-{\bf T}_{h}{\bf E}\right),\epsilon q_{h}\rangle_{\partial K}.\label{erreqn6}
		\end{eqnarray}
		Since $q\in Q_{h}^{0}$, we have exploited this fact in the last equality of \eqref{erreqn6} to have
		\begin{equation*}
			\sum_{K\in\mathcal{K}_{h}}\langle{\bf n}\cdot{\bf E},\epsilon\avg{q_{h}}\rangle_{\partial K}=0.
		\end{equation*}
		Subtracting \eqref{sdalg2} from \eqref{erreqn6}, yields the desired second error equation in \eqref{sderrorequation}.
	\end{proof}
	\begin{lemma}[cf. Lemma 4.11, \cite{mohapatra2025new}]\label{bounds}
		Assume that the true solution of model \eqref{model3},
		\[
		({\bf E}(t),P(t))\in 
		\bigl({\bf H}^{1+k}(\Omega)\cap {\bf H}_0(\mathrm{curl};\Omega)\bigr)
		\times
		\bigl(H^{k}(\Omega)\cap H^{1}_{0}(\Omega)\bigr).
		\]
		Then, for arbitrary \(({\bf v},q_h)\in {\bf V}_{h}^{0}\times Q_{h}^{0}\), the following bounds for the residuals hold,
		\begin{eqnarray*}
			|\mathcal{M}_l({\bf E},{\bf v}_h)|&\leq& Ch^{k}\norm{{\bf E}}_{1+k}\tnorm{{\bf v}_h},\;\;l=1,\;2,\\
			|\mathcal{M}_3(P,{\bf v}_h)|&\leq& Ch^{k}\norm{P}_{k}\tnorm{{\bf v}_h},\\
			|\mathcal{M}_4({\bf E},q_h)|&\leq& Ch^{1+k}\norm{{\bf E}}_{1+k}\tnorm{q_{h}}_{0},\\
			|\mathcal{S}({\bf T}_h{\bf E},{\bf v}_h)|&\leq& Ch^{k}\norm{{\bf E}}_{1+k}\tnorm{{\bf v}_h}.
		\end{eqnarray*}
	\end{lemma}
	\begin{thm}\label{SemHcurl}
		Under the regularity assumptions:
		\begin{eqnarray*}
			&&{\bf E}\in {\bf H}^{3}\!\left({\bf H}^{1+k}(\Omega)\cap {\bf H}_0(\mathrm{curl};\Omega)\right),\\
			&&P\in H^{2}\!\left(H^{k}(\Omega)\cap H^{1}_{0}(\Omega)\right),
		\end{eqnarray*}
		of the true solution of \eqref{model3}, the following error estimates hold:
		\begin{eqnarray}\label{semiest1}
			\norm{\epsilon^{\frac{1}{2}}{\bf e}_{ht}(t)}^{2}
			+\tnorm{{\bf e}_h(t)}^{2}
			\le
			C\,h^{2k}
			\left(
			\norm{{\bf E}}_{{\bf H}^{3}({\bf H}^{1+k}(\Omega))}^{2}
			+\norm{P}_{H^{2}(H^{k}(\Omega))}^{2}
			\right),
		\end{eqnarray}
		and
		\begin{eqnarray}\label{semiest2}
			\norm{\tilde{e}_{h}(t)}^{2}
			\le
			C\,h^{2k}
			\left(
			\norm{{\bf E}}_{{\bf H}^{3}({\bf H}^{1+k}(\Omega))}^{2}
			+\norm{P}_{H^{2}(H^{k}(\Omega))}^{2}
			\right).
		\end{eqnarray}
	\end{thm}
	\begin{proof}
		We start the proof by differentiating the second equation of \eqref{sderrorequation} with respect to $t$, to have 
		\begin{equation}\label{err0}
			B({\bf e}_{ht},q_{h}) = \mathcal{M}_4({\bf E}_{t},q_{h})\,\forall q_{h}\in Q_{h}^{0}.
		\end{equation}
		Now, substituting $q_{h}=\tilde{e}_{h}$ in above equation and adding with the first error equation of \eqref{sderrorequation} by selecting the test function ${\bf v}_{h}={\bf e}_{ht}$, we obtain
		\begin{eqnarray*}
			&&	\frac{1}{2}\frac{d}{dt}\left(\norm{\epsilon^{\frac{1}{2}}{\bf e}_{ht}}^{2}+\tnorm{{\bf e}_h}^{2}\right)
			\nonumber\\
			&&~~=\sum_{i=1}^{2}\mathcal{M}_i({\bf E},{\bf e}_{ht})+\mathcal{M}_3(P,{\bf e}_{ht})+\mathcal{S}({\bf T}_h{\bf E},{\bf e}_{ht})+
			\mathcal{M}_4({\bf E}_{t},\tilde{e}_{h}).\nonumber\\
		\end{eqnarray*}
		Integrating both sides of above equation from $0$ to $t\, (0< t\le T)$ with respect to $s$, we achieve
		\begin{eqnarray}\label{err1}
			&&\norm{\epsilon^{\frac{1}{2}}{\bf e}_{ht}(t)}^{2}+\tnorm{{\bf e}_h(t)}^{2}\nonumber\\
			&&~~\le\norm{\epsilon^{\frac{1}{2}}{\bf e}_{ht}(0)}^{2}+\tnorm{{\bf e}_h(0)}^{2}\nonumber\\
			&&~~~~+2\int_{0}^{t}\left(\sum_{i=1}^{2}\mathcal{M}_i({\bf E},{\bf e}_{hs})+\mathcal{M}_3(P,{\bf e}_{hs})+\mathcal{S}({\bf T}_h{\bf E},{\bf e}_{hs})+
			\mathcal{M}_4({\bf E}_{s},\tilde{e}_{h})\right)ds.~~~~~~~~~
		\end{eqnarray}
		Next, we proceed to find and estimate for each term on the right side of \eqref{err1}. 
		
		From Lemma \ref{ritzestimate}, it follows that
		\begin{eqnarray}\label{err2*}
			\norm{{\bf e}_{ht}(0)}&=&\norm{({\bf T}_{h}{\bf E}-{\bf E}_{h})_{t}(0)}\nonumber\\
			&=&\norm{{\bf T}_{h}{\bf E}_{t}(0)-{\bf E}_{ht}(0)}\nonumber\\
			&=&\norm{{\bf T}_{h}{\bf v}-\mathbb{R}_{h}{\bf v}}\nonumber\\
			&\le&Ch^{1+k}\left(\norm{{\bf v}}_{1+k}+\norm{P_{t}(0)}_{k}\right)\nonumber\\
			&\le&Ch^{1+k}\left(\norm{{\bf v}}_{1+k}+\max_{t\in [0,T]}\norm{P_{t}(t)}_{k}\right).
		\end{eqnarray}
		Arguing in the same way, we have
		\begin{eqnarray}\label{err3*}
			\tnorm{{\bf e}_{h}(0)}&\le&\tnorm{{\bf e}_{h}(0)}_{1}\nonumber\\
			&=&\tnorm{({\bf T}_{h}{\bf E}-{\bf E}_{h})(0)}_{1}\nonumber\\
			&=&\tnorm{{\bf T}_{h}{\bf E}(0)-{\bf E}_{h}(0)}_{1}\nonumber\\
			&=&\tnorm{{\bf T}_{h}{\bf u}-\mathbb{R}_{h}{\bf u}}_{1}\nonumber\\
			&\le&Ch^{k}\left(\norm{{\bf u}}_{1+k}+\norm{P(0)}_{k}\right)\nonumber\\
			&\le&Ch^{k}\left(\norm{{\bf u}}_{1+k}+\max_{t\in[0,T]}\norm{P(t)}_{k}\right).
		\end{eqnarray}
		Using integration by parts, Lemma \ref{bounds} and estimate \eqref{err3*}, for $i=1,2$, we  observe
		\begin{eqnarray}\label{err3}
			\left|\int_{0}^{t}\mathcal{M}_i({\bf E},{\bf e}_{hs})ds\right|&=&\left|\mathcal{M}_i({\bf E}(t),{\bf e}_{h}(t))-\mathcal{M}_i({\bf E}(0),{\bf e}_{h}(0))-\int_{0}^{t}\mathcal{M}_i({\bf E}_{s},{\bf e}_{h})ds \right|\nonumber\\
			&\le&\left|\mathcal{M}_i({\bf E}(t),{\bf e}_{h}(t))\right|+\left|\mathcal{M}_i({\bf E}(0),{\bf e}_{h}(0))\right|+\int_{0}^{t}|\mathcal{M}_i({\bf E}_{s},{\bf e}_{h})|ds\nonumber\\
			&\le&Ch^{k}\norm{{\bf E}}_{1+k}\tnorm{{\bf e}_{h}}+Ch^{k}\norm{{\bf u}}_{1+k}\tnorm{{\bf e}_{h}(0)}\nonumber\\
			&&+Ch^{k}\int_{0}^{t}\norm{{\bf E}_{s}}_{1+k}\tnorm{{\bf e}_{h}}ds\nonumber\\
			&\le&Ch^{k}\norm{{\bf E}}_{1+k}\tnorm{{\bf e}_{h}}+Ch^{2k}\norm{{\bf u}}_{1+k}(\norm{{\bf u}}_{1+k}+\max_{t\in[0,T]}\norm{P(t)}_{k})\nonumber\\
			&&+Ch^{k}\int_{0}^{t}\norm{{\bf E}_{s}}_{1+k}\tnorm{{\bf e}_{h}}ds.
		\end{eqnarray}
		Arguing similarly, we can derive the following bounds
		\begin{eqnarray}
			\left|\int_{0}^{t}\mathcal{M}_3(P,{\bf e}_{hs})ds\right|&\le& Ch^{k}\norm{P}_{k}\tnorm{{\bf e}_{h}}+Ch^{2k}\max_{t\in [0,T]}\norm{P(t)}_{k}(\norm{{\bf u}}_{1+k}+\max_{t\in[0,T]}\norm{P(t)}_{k})\nonumber\\
			&&+Ch^{k}\int_{0}^{t}\norm{P_{s}}_{k}\tnorm{{\bf e}_{h}}ds,\label{err4}\\
			\left|\int_{0}^{t}\mathcal{S}({\bf T}_h{\bf E},{\bf e}_{hs})ds\right| &\le& Ch^{k}\norm{{\bf E}}_{1+k}\tnorm{{\bf e}_{h}}+Ch^{2k}\norm{{\bf u}}_{1+k}(\norm{{\bf u}}_{1+k}+\max_{t\in[0,T]}\norm{P(t)}_{k})\nonumber\\
			&&+Ch^{k}\int_{0}^{t}\norm{{\bf E}_{s}}_{1+k}\tnorm{{\bf e}_{h}}ds,\label{err5}\\
			\left|\int_{0}^{t}\mathcal{M}_4({\bf E}_{s},\tilde{e}_{h})ds\right|&\le&
			Ch^{1+k}\int_{0}^{t}\norm{{\bf E}_{s}}_{1+k}\tnorm{{\tilde{e}_{h}}}_{0}ds\label{err6}.
		\end{eqnarray}
		Plugging the bounds \eqref{err2*}-\eqref{err6} in \eqref{err1}, we achieve
		\begin{eqnarray}\label{err7}
			&&\norm{\epsilon^{\frac{1}{2}}{\bf e}_{ht}(t)}^{2}+\tnorm{{\bf e}_h(t)}^{2}\nonumber\\
			&&~~\le Ch^{2(1+k)}\left(\norm{{\bf v}}_{1+k}^{2}+\max_{t\in[0,T]}\norm{P_{t}(t)}_{k}^{2}\right)+Ch^{2k}\left(\norm{{\bf u}}_{1+k}^{2}+\max_{t\in[0,T]}\norm{P(t)}_{k}^{2}\right)\nonumber\\
			&&~~~~+Ch^{k}\norm{{\bf E}}_{1+k}\tnorm{{\bf e}_{h}}+Ch^{k}\norm{P}_{k}\tnorm{{\bf e}_{h}}\nonumber\\
			&&~~~~+Ch^{k}\int_{0}^{t}\norm{{\bf E}_{s}}_{1+k}\tnorm{{\bf e}_{h}}ds+Ch^{k}\int_{0}^{t}\norm{P_{s}}_{k}\tnorm{{\bf e}_{h}}ds\nonumber\\
			&&~~~~+Ch^{1+k}\int_{0}^{t}\norm{{\bf E}_{s}}_{1+k}\tnorm{\tilde{e}_{h}}_{0}ds.
		\end{eqnarray}
		Applying Young's inequality and utilizing the inverse estimate in $\tnorm{\cdot}_{0}$-norm (cf. Lemma \ref{energyinverse}) in \eqref{err7} leads to
		\begin{eqnarray}\label{err8}
			&&\norm{\epsilon^{\frac{1}{2}}{\bf e}_{ht}(t)}^{2}+\frac{1}{2}\tnorm{{\bf e}_h(t)}^{2}\nonumber\\
			&&~~\le Ch^{2k}\left(\norm{{\bf u}}_{1+k}^{2}+\norm{{\bf v}}_{1+k}^{2}+\max_{t\in[0,T]}\norm{P(t)}_{k}^{2}\right.\nonumber\\
			&&~~~~~~~~~~~~~~\left.+\max_{t\in[0,T]}\norm{P_{t}(t)}_{k}^{2}+\max_{t\in[0,T]}\norm{{\bf E}(t)}^{2}_{1+k}+\max_{t\in[0,T]}\norm{P(t)}_{k}^{2}\right)\nonumber\\
			&&~~~~+Ch^{2k}\int_{0}^{t}\norm{{\bf E}_{s}}_{1+k}^{2}ds+Ch^{2k}\int_{0}^{t}\norm{P_{s}}_{k}^{2}ds+C\int_{0}^{t}\tnorm{{\bf e}_{h}}^{2}ds+C\int_{0}^{t}\norm{\tilde{e}_{h}}^{2}ds\nonumber\\
			&&~~\le Ch^{2k}\left(\norm{P}_{H^{2}(H^{k}(\Omega))}^{2}+\norm{{\bf E}}^{2}_{{\bf H}^{3}({\bf H}^{1+k}(\Omega))}\right)+C\int_{0}^{t}\norm{\tilde{e}_{h}}^{2}ds+C\int_{0}^{t}\tnorm{{\bf e}_{h}}^{2}ds.~~~~~~~~
		\end{eqnarray}
		From continuous Gronwall's inequality 1 (cf. Lemma \ref{g1}) and by standard Sobolev embedding results, we get
		\begin{eqnarray}\label{err88}
			&&\norm{\epsilon^{\frac{1}{2}}{\bf e}_{ht}(t)}^{2}+\tnorm{{\bf e}_h(t)}^{2}\nonumber\\
			&&~~\le Ch^{2k}\left(\norm{P}_{H^{2}(H^{k})}^{2}+\norm{{\bf E}}^{2}_{{\bf H}^{3}({\bf H}^{1+k})}\right)+C\int_{0}^{t}\norm{\tilde{e}_{h}}^{2}ds.
		\end{eqnarray}
		Now, we have to find an estimate for the term $\norm{\tilde{e}_{h}}$. For this we need some preparations. At first, we determine an estimate for the term $\norm{\epsilon^{\frac{1}{2}}{\bf e}_{htt}}$. We begin by differentiating the first error equation of \eqref{sderrorequation} and \eqref{err0} with respect to $t$, for any $({\bf v}_{h},q_{h}) \in {\bf V}_{h}^{0}\times Q_{h}^{0}$, to obtain 
		\begin{eqnarray*}
			\left\{
			\begin{array}{ll}
				(\epsilon{\bf e}_{httt},{\bf v}_h)+A({\bf e}_{ht},{\bf v}_h)-B({\bf v}_{h},\tilde{e}_{ht})=\sum_{i=1}^{2}\mathcal{M}_i({\bf E}_{t},{\bf v}_{h})+\mathcal{M}_3(P_{t},{\bf v}_{h})+\mathcal{S}({\bf T}_h{\bf E}_{t},{\bf v}_h),\\
				B({\bf e}_{htt},q_{h}) = \mathcal{M}_4({\bf E}_{tt},q_{h}).
			\end{array}
			\right.
		\end{eqnarray*}
		In the above system replacing $q_h=\tilde e_{ht}$ and choosing ${\bf v}_h={\bf e}_{htt}$, then later adding the resulting equations, we obtain
		\begin{eqnarray*}
			&&	\frac{1}{2}\frac{d}{dt}\left(\norm{\epsilon^{\frac{1}{2}}{\bf e}_{htt}}^{2}+\tnorm{{\bf e}_{ht}}^{2}\right)
			\nonumber\\
			&&~~=\sum_{i=1}^{2}\mathcal{M}_i({\bf E}_{t},{\bf e}_{htt})+\mathcal{M}_3(P_{t},{\bf e}_{htt})+\mathcal{S}({\bf T}_h{\bf E}_{t},{\bf e}_{htt})+
			\mathcal{M}_4({\bf E}_{tt},\tilde{e}_{ht}).
		\end{eqnarray*}
		Integrating both sides of above equation from $0$ to $t\, (0\,< t\le T)$ with respect to $s$, we achieve
		\begin{eqnarray}\label{err9}
			&&\norm{\epsilon^{\frac{1}{2}}{\bf e}_{htt}(t)}^{2}+\tnorm{{\bf e}_{ht}(t)}^{2}\nonumber\\
			&&~~\le\norm{\epsilon^{\frac{1}{2}}{\bf e}_{htt}(0)}^{2}+\tnorm{{\bf e}_{ht}(0)}^{2}\nonumber\\
			&&~~~~+2\int_{0}^{t}\left(\sum_{i=1}^{2}\mathcal{M}_i({\bf E}_{s},{\bf e}_{hss})+\mathcal{M}_3(P_{s},{\bf e}_{hss})+\mathcal{S}({\bf T}_h{\bf E}_{s},{\bf e}_{hss})+
			\mathcal{M}_4({\bf E}_{ss},\tilde{e}_{hs})\right)ds.~~~~~~~~~
		\end{eqnarray}
		From \eqref{initialsecond}, it is easy to see ${\bf e}_{htt}(0)=0$. Now,
		arguing as in derivation of \eqref{err2*}, then by Lemma \ref{bounds}, we deduce
		\begin{eqnarray}
			\tnorm{{\bf e}_{ht}(0)}&\le&Ch^{k}\left(\norm{{\bf v}}_{1+k}+\max_{t\in [0,T]}\norm{P_{t}(t)}_{k}\right).\label{err10}
		\end{eqnarray}
		Using integration by parts, Lemma \ref{bounds} and estimate \eqref{err10}, for $i=1,2$, we  observe
		\begin{eqnarray}\label{err11}
			\left|\int_{0}^{t}\mathcal{M}_i({\bf E}_{s},{\bf e}_{hss})ds\right|&=&\left|\mathcal{M}_i({\bf E}_{t}(t),{\bf e}_{ht}(t))-\mathcal{M}_i({\bf E}_{t}(0),{\bf e}_{ht}(0))-\int_{0}^{t}\mathcal{M}_i({\bf E}_{ss},{\bf e}_{hs})ds \right|\nonumber\\
			&\le&\left|\mathcal{M}_i({\bf E}_{t}(t),{\bf e}_{ht}(t))\right|+\left|\mathcal{M}_i({\bf E}_{t}(0),{\bf e}_{ht}(0))\right|+\int_{0}^{t}|\mathcal{M}_i({\bf E}_{ss},{\bf e}_{hs})|ds\nonumber\\
			&\le&Ch^{k}\norm{{\bf E}_{t}}_{1+k}\tnorm{{\bf e}_{ht}}+Ch^{k}\norm{{\bf v}}_{1+k}\tnorm{{\bf e}_{ht}(0)}\nonumber\\
			&&+Ch^{k}\int_{0}^{t}\norm{{\bf E}_{ss}}_{1+k}\tnorm{{\bf e}_{hs}}ds\nonumber\\
			&\le&Ch^{k}\norm{{\bf E}_{t}}_{1+k}\tnorm{{\bf e}_{ht}}+Ch^{2k}\norm{{\bf v}}_{1+k}(\norm{{\bf v}}_{1+k}+\max_{t\in[0,T]}\norm{P_{t}(t)}_{k})\nonumber\\
			&&+Ch^{k}\int_{0}^{t}\norm{{\bf E}_{ss}}_{1+k}\tnorm{{\bf e}_{hs}}ds.
		\end{eqnarray}
		In the same manner, we have the following estimates
		\begin{eqnarray}
			\left|\int_{0}^{t}\mathcal{M}_3(P_{s},{\bf e}_{hss})ds\right|&\le& Ch^{k}\norm{P_{t}}_{k}\tnorm{{\bf e}_{ht}}+Ch^{2k}\max_{t\in [0,T]}\norm{P_{t}(t)}_{k}(\norm{{\bf v}}_{1+k}+\max_{t\in[0,T]}\norm{P_{t}(t)}_{k})\nonumber\\
			&&+Ch^{k}\int_{0}^{t}\norm{P_{ss}}_{k}\tnorm{{\bf e}_{h}}ds,\label{err12}\\
			\left|\int_{0}^{t}\mathcal{S}({\bf T}_h{\bf E}_{s},{\bf e}_{hss})ds\right| 	&\le&Ch^{k}\norm{{\bf E}_{t}}_{1+k}\tnorm{{\bf e}_{ht}}+Ch^{2k}\norm{{\bf v}}_{1+k}(\norm{{\bf v}}_{1+k}+\max_{t\in[0,T]}\norm{P_{t}(t)}_{k})\nonumber\\
			&&+Ch^{k}\int_{0}^{t}\norm{{\bf E}_{ss}}_{1+k}\tnorm{{\bf e}_{hs}}ds,	\label{err13}\\
			\left|\int_{0}^{t}\mathcal{M}_4({\bf E}_{ss},\tilde{e}_{hs})ds\right|&\le& Ch^{1+k}\norm{{\bf E}_{tt}}_{1+k}\tnorm{{\tilde{e}_{h}}}_{0}+Ch^{1+k}\norm{{\bf E}_{tt}(0)}_{1+k}\tnorm{{\tilde{e}_{h}}(0)}_{0}\nonumber\\
			&&+Ch^{1+k}\int_{0}^{t}\norm{{\bf E}_{sss}}_{1+k}\tnorm{{\tilde{e}_{h}}}_{0}ds\nonumber\\
			&\le& Ch^{k}\norm{{\bf E}_{tt}}_{1+k}\norm{{\tilde{e}_{h}}}+Ch^{k}\norm{{\bf E}_{tt}(0)}_{1+k}\norm{{\tilde{e}_{h}}(0)}\nonumber\\
			&&+Ch^{k}\int_{0}^{t}\norm{{\bf E}_{sss}}_{1+k}\norm{{\tilde{e}_{h}}}ds\label{err14}.
		\end{eqnarray}

		Note that in \eqref{err14}, we have used integration by parts, Lemma \ref{bounds}, followed by the use of inverse estimate in $\tnorm{\cdot}_{0}$-norm (cf. Lemma \ref{energyinverse}).
		
		Combining estimates \eqref{err10}-\eqref{err14}, then using them in \eqref{err9} and by Lemma \ref{ritzestimate}, we deduce
		\begin{eqnarray*}
			&&\norm{\epsilon^{\frac{1}{2}}{\bf e}_{htt}(t)}^{2}+\tnorm{{\bf e}_{ht}(t)}^{2}\nonumber\\
			&&~~\le Ch^{2k}\left(\norm{{\bf v}}_{1+k}^{2}+\max_{t\in[0,T]}\norm{P_{t}(t)}_{k}^{2}\right)\nonumber\\
			&&~~~~+Ch^{k}\norm{{\bf E}_{t}}_{1+k}\tnorm{{\bf e}_{ht}}+Ch^{2k}\norm{{\bf v}}_{1+k}(\norm{{\bf v}}_{1+k}+\max_{t\in[0,T]}\norm{P_{t}(t)}_{k})\nonumber\\
			&&~~~~+Ch^{k}\norm{P_{t}}_{k}\tnorm{{\bf e}_{ht}}+Ch^{2k}\max_{t\in [0,T]}\norm{P_{t}(t)}_{k}(\norm{{\bf v}}_{1+k}+\max_{t\in[0,T]}\norm{P_{t}(t)}_{k})\nonumber\\
			&&~~~~+Ch^{k}\norm{{\bf E}_{tt}}_{1+k}\norm{{\tilde{e}_{h}}}+Ch^{2k}\norm{{\bf E}_{tt}(0)}_{1+k}(\norm{{\bf u}}_{1+k}+\max_{t\in[0,T]}\norm{P(t)}_{k})\nonumber\\
			&&~~~~+Ch^{k}\int_{0}^{t}\norm{{\bf E}_{ss}}_{1+k}\tnorm{{\bf e}_{hs}}ds+Ch^{k}\int_{0}^{t}\norm{P_{ss}}_{k}\tnorm{{\bf e}_{hs}}ds\nonumber\\
			&&~~~~+Ch^{k}\int_{0}^{t}\norm{{\bf E}_{sss}}_{1+k}\norm{\tilde{e}_{h}}ds.
		\end{eqnarray*}
		Applying Young's inequality, we further derive
		\begin{eqnarray*}
			&&\norm{\epsilon^{\frac{1}{2}}{\bf e}_{htt}(t)}^{2}+\tnorm{{\bf e}_{ht}(t)}^{2}\nonumber\\
			&&~~\le Ch^{2k}\left(\norm{{\bf v}}_{1+k}^{2}+\max_{t\in[0,T]}\norm{P_{t}(t)}_{k}^{2}\right)+Ch^{2k}\norm{{\bf E}_{t}}_{1+k}^{2}+\frac{1}{4}\tnorm{{{\bf e}_{ht}}}^{2}\nonumber\\
			&&~~~~+Ch^{2k}\norm{{\bf v}}_{1+k}^{2}+Ch^{2k}\left(\norm{{\bf v}}_{1+k}^{2}+\max_{t\in[0,T]}\norm{P_{t}(t)}_{k}^{2}\right)\nonumber\\
			&&~~~~+Ch^{2k}\norm{P_{t}}_{k}^{2}+\frac{1}{4}\tnorm{{\bf e}_{ht}}^{2}+Ch^{2k}\max_{t\in[0,T]}\norm{P_{t}(t)}^{2}+Ch^{2k}\left(\norm{{\bf v}}_{1+k}^{2}+\max_{t\in[0,T]}\norm{P_{t}(t)}_{k}^{2}\right)\nonumber\\
			&&~~~~+Ch^{k}\norm{{\bf E}_{tt}}_{1+k}\norm{{\tilde{e}_{h}}}+Ch^{2k}\norm{{\bf E}_{tt}(0)}_{1+k}^{2}+Ch^{2k}\left(\norm{{\bf u}}_{1+k}^{2}+\max_{t\in[0,T]}\norm{P(t)}_{k}^{2}\right)\nonumber\\
			&&~~~~+Ch^{2k}\int_{0}^{t}\norm{{\bf E}_{ss}}_{1+k}^{2}ds+\frac{1}{2}\int_{0}^{t}\tnorm{{\bf e}_{hs}}^{2}ds+Ch^{2k}\int_{0}^{t}\norm{P_{ss}}_{k}^{2}ds\nonumber\\
			&&~~~~+\frac{1}{2}\int_{0}^{t}\tnorm{{\bf e}_{hs}}^{2}ds+Ch^{2k}\int_{0}^{t}\norm{{\bf E}_{sss}}_{1+k}^{2}ds+\frac{1}{2}\int_{0}^{t}\norm{\tilde{e}_{h}}^{2}ds\nonumber\\
			&&~~\le Ch^{2k}\left(\norm{{\bf E}}_{{\bf H}^{3}({\bf H}^{1+k}(\Omega))}^{2}+\norm{P}_{H^{2}(H^{k}(\Omega))}^{2}\right)+Ch^{k}\norm{{\bf E}_{tt}}_{1+k}\norm{{\tilde{e}_{h}}}+\frac{1}{2}\tnorm{{\bf e}_{ht}}^{2}\nonumber\\
			&&~~~~+C\int_{0}^{t}\tnorm{{\bf e}_{hs}}^{2}ds+C\int_{0}^{t}\norm{\tilde{e}_{h}}^{2}ds.
		\end{eqnarray*}
		Simplifying the above inequality results in
		\begin{eqnarray*}
			\norm{\epsilon^{\frac{1}{2}}{\bf e}_{htt}(t)}^{2}+\tnorm{{\bf e}_{ht}(t)}^{2}
			&\le& Ch^{2k}\left(\norm{{\bf E}}_{{\bf H}^{3}({\bf H}^{1+k}(\Omega))}^{2}+\norm{P}_{H^{2}(H^{k}(\Omega))}^{2}\right)+Ch^{k}\norm{{\bf E}_{t}}_{1+k}\norm{{\tilde{e}_{h}}}\nonumber\\
			&&+C\int_{0}^{t}\tnorm{{\bf e}_{hs}}^{2}ds+C\int_{0}^{t}\norm{\tilde{e}_{h}}^{2}ds.
		\end{eqnarray*}
		Utilizing Modified Gronwall's inequality (cf. Lemma \ref{g2}) in the above estimate, we observe
		\begin{eqnarray}\label{err15}
			\norm{\epsilon^{\frac{1}{2}}{\bf e}_{htt}(t)}^{2}+\tnorm{{\bf e}_{ht}(t)}^{2}
			&\le& Ch^{2k}\left(\norm{{\bf E}}_{{\bf H}^{3}({\bf H}^{1+k}(\Omega))}^{2}+\norm{P}_{H^{2}(H^{k}(\Omega))}^{2}\right)+Ch^{k}\norm{{\bf E}_{t}}_{1+k}\norm{{\tilde{e}_{h}}}\nonumber\\
			&&+Ch^{k}\int_{0}^{t}\norm{{\bf E}_{s}}_{1+k}\norm{{\tilde{e}_{h}}}ds+C\int_{0}^{t}\norm{\tilde{e}_{h}}^{2}ds.
		\end{eqnarray}
		
		It follows from the second error equation of \eqref{sderrorequation} and the inf-sup condition (cf. Lemma \ref{infsup}) that for $\tilde{e}_{h}\in Q_{h}^{0}$ there exists some ${\bf v}_{h}\in {\bf V}_{h}^{0}$ such that we have
		\begin{eqnarray*}
			\norm{\tilde{e}_{h}}^{2}&=&B({\bf v}_{h},\tilde{e}_{h})\nonumber\\
			&=&(\epsilon{\bf e}_{htt},{\bf v}_h)+A({\bf e}_h,{\bf v}_h)-\sum_{i=1}^{2}\mathcal{M}_i({\bf E},{\bf v}_{h})-\mathcal{M}_3(P,{\bf v}_{h})-\mathcal{S}({\bf T}_h{\bf E},{\bf v}_h)\nonumber\\
			&\le& C\norm{\epsilon^{\frac{1}{2}}{\bf e}_{htt}}\norm{{\bf v}_{h}}+C\tnorm{{\bf e}_{h}}\tnorm{{\bf v}_{h}}+Ch^{k}\norm{{\bf E}}_{1+k}\tnorm{{\bf v}_{h}}+Ch^{k}\norm{P}_{k}\tnorm{{\bf v}_{h}}\nonumber\\
			&\le& C\norm{\epsilon^{\frac{1}{2}}{\bf e}_{htt}}\tnorm{{\bf v}_{h}}_{1}+C\tnorm{{\bf e}_{h}}\tnorm{{\bf v}_{h}}_{1}+Ch^{k}\norm{{\bf E}}_{1+k}\tnorm{{\bf v}_{h}}_{1}+Ch^{k}\norm{P}_{k}\tnorm{{\bf v}_{h}}_{1}\nonumber\\
			&\le& C\norm{\epsilon^{\frac{1}{2}}{\bf e}_{htt}}\norm{\tilde{e}_{h}}+\tnorm{{\bf e}_{h}}\norm{\tilde{e}_{h}}+Ch^{k}\norm{{\bf E}}_{1+k}\norm{\tilde{e}_{h}}+Ch^{k}\norm{P}_{k}\norm{\tilde{e}_{h}}\nonumber\\
			&\le& C\norm{\epsilon^{\frac{1}{2}}{\bf e}_{htt}}^{2}+C\tnorm{{\bf e}_{h}}^{2}+Ch^{2k}\norm{{\bf E}}_{1+k}^{2}+Ch^{2k}\norm{P}_{k}^{2}+\frac{1}{2}\norm{\tilde{e}_{h}}^{2}.
		\end{eqnarray*}
		Note that for deriving the above bound, we have used Cauchy-Schwarz, Lemma \ref{bounds}, the fact $\tnorm{{\bf v}_{h}}_{1}\le C\norm{\tilde{e}_{h}}$ (cf. Lemma \ref{infsup}), Lemma \ref{bounded} and Young's inequality.
		
		From \eqref{err88} and \eqref{err15}, we deduce
		\begin{eqnarray}
			\norm{\tilde{e}_{h}}^{2}&\le& C\norm{\epsilon^{\frac{1}{2}}{\bf e}_{htt}}^{2}+C\tnorm{{\bf e}_{h}}^{2}+Ch^{2k}\norm{{\bf E}}_{1+k}^{2}+Ch^{2k}\norm{P}_{k}^{2}\nonumber\\
			&\le& Ch^{2k}\left(\norm{P}_{H^{2}(H^{k}(\Omega))}^{2}+\norm{{\bf E}}^{2}_{{\bf H}^{3}({\bf H}^{1+k}(\Omega))}\right)+C\int_{0}^{t}\norm{\tilde{e}_{h}}^{2}ds\nonumber\\
			&&+ Ch^{2k}\left(\norm{{\bf E}}_{{\bf H}^{3}({\bf H}^{1+k})}^{2}+\norm{P}_{H^{2}(H^{k})}^{2}\right)+Ch^{k}\norm{{\bf E}_{t}}_{1+k}\norm{{\tilde{e}_{h}}}\nonumber\\
			&&+Ch^{k}\int_{0}^{t}\norm{{\bf E}_{s}}_{1+k}\norm{{\tilde{e}_{h}}}ds+C\int_{0}^{t}\norm{\tilde{e}_{h}}^{2}ds\nonumber\\
			&&+Ch^{2k}\norm{{\bf E}}_{1+k}^{2}+Ch^{2k}\norm{P}_{k}^{2}\nonumber\\
			&\le& Ch^{2k}\left(\norm{P}_{H^{2}(H^{k}(\Omega))}^{2}+\norm{{\bf E}}^{2}_{{\bf H}^{3}({\bf H}^{1+k}(\Omega))}\right)+Ch^{2k}\norm{{\bf E}_{t}}_{1+k}^{2}+\frac{1}{2}\norm{{\tilde{e}_{h}}}^{2}\nonumber\\
			&&+Ch^{2k}\int_{0}^{t}\norm{{\bf E}_{s}}_{1+k}^{2}ds+C\int_{0}^{t}\norm{\tilde{e}_{h}}^{2}ds.
		\end{eqnarray}	
		Hence, by Gronwall's inequality 1 (cf. Lemma \ref{g1}), we achieve
		\begin{eqnarray}\label{err16}
			\norm{\tilde{e}_{h}}^{2}&\le& Ch^{2k}\left(\norm{P}_{H^{2}(H^{k}(\Omega))}^{2}+\norm{{\bf E}}^{2}_{{\bf H}^{3}({\bf H}^{1+k}(\Omega))}\right).
		\end{eqnarray}
		This completes the proof of \eqref{semiest2}. Utilizing \eqref{semiest2} in \eqref{err88} establishes \eqref{semiest1}.
	\end{proof}
	\begin{rem}
		From triangle inequality, \eqref{ritzestimate} and \eqref{semiest1}, we achieve
		\begin{eqnarray}
			\norm{\epsilon^{\frac{1}{2}}{\bf e}_{h}(t)}^{2}
			&\le&C\norm{\int_{0}^{t}\epsilon^{\frac{1}{2}}{\bf e}_{hs}(s)ds}^{2}+C\norm{\epsilon^{\frac{1}{2}}{\bf e}_{h}(0)}^{2}\nonumber\\
			&\le&CT\norm{\epsilon^{\frac{1}{2}}{\bf e}_{ht}(t)}_{{\bf L}^{\infty}({\bf L}^{2}(\Omega))}^{2}+C\norm{\epsilon^{\frac{1}{2}}({\bf T}_{h}{\bf E}(0)-{\bf E}_{h}(0))}^{2}\nonumber\\
			&\le&CT\norm{\epsilon^{\frac{1}{2}}{\bf e}_{ht}(t)}_{{\bf L}^{\infty}({\bf L}^{2}(\Omega))}^{2}+Ch^{2(1+k)}\left(\norm{{\bf u}}_{1+k}^{2}+\norm{P(0)}^{2}_{k}\right)\nonumber\\
			&\le&CT\norm{\epsilon^{\frac{1}{2}}{\bf e}_{ht}(t)}_{{\bf L}^{\infty}({\bf L}^{2}(\Omega))}^{2}+Ch^{2(1+k)}\left(\norm{{\bf u}}_{1+k}^{2}+\max_{t\in[0,T]}\norm{P(t)}^{2}_{k}\right)\nonumber\\
			&\le&	 Ch^{2k}\left(\norm{{\bf E}}_{{\bf H}^{3}({\bf H}^{1+k}(\Omega))}^{2}+\norm{P}_{H^{2}(H^{k}(\Omega))}^{2}\right).
		\end{eqnarray}
	\end{rem}
	The above estimate combined with \eqref{semiest1} results in the following optimal convergence result for the error under the discrete energy norm given below.
	\begin{cor}\label{fullest}
		Let the assumptions of Theorem \ref{SemHcurl} hold true. Then we have the following error estimate in energy norm
		\begin{eqnarray}\label{semiest3}
			\mathrm{ess}\sup_{t\in[0,T]}\tnorm{{\bf e}_{h}(t)}_{1}^{2}
			&\le&	 Ch^{2k}\left(\norm{{\bf E}}_{{\bf H}^{3}({\bf H}^{1+k}(\Omega))}^{2}+\norm{P}_{H^{2}(H^{k}(\Omega))}^{2}\right).
		\end{eqnarray}
	\end{cor}
	\begin{rem}
		Under similar regularity assumptions, the authors in \cite{qi2025decoupled} established error estimates in only an ``energy" semi-norm for the electric field error. However, a convergence result for the potential function's error in the \(L^{2}\)-norm was also provided.
	\end{rem}
	Our objective now remains to determine the ${\bf L}^{2}$-norm error estimate for the electric field variable. We now split the semi-discrete projected errors ${\bf e}_{h}$ and $\tilde{e}_{h}$ in the following way:
	\begin{eqnarray}
		{\bf e}_{h}={\bf T}_{h}{\bf E}-{\bf E}_{h}&:=&\boldsymbol{\rho}_{h}+\boldsymbol{\chi}_{h},\nonumber\\
		\tilde{e}_{h}=L_{h}P-P_{h}&:=&\delta_{h}+\gamma_{h}.
	\end{eqnarray}
	Here $\boldsymbol{\rho}_{h}={\bf T}_{h}{\bf E}-\mathbb{R}_{h}{\bf E}$, $\boldsymbol{\chi}_{h}=\mathbb{R}_{h}{\bf E}-{\bf E}_{h}$, $\delta_{h}=L_{h}P-\mathbb{E}_{h}P$ and $\gamma_{h}=\mathbb{E}_{h}P-P_{h}$. 
	As a consequence of Lemma \ref{ritzestimate}, for the case $({\bf z},q)=({\bf E},P)$, we have the following estimate for the term $\boldsymbol{\rho}_{h}$ given as
	\begin{eqnarray}\label{rho}
		\norm{{\bf T}_{h}{\bf E}-\mathbb{R}_{h}{\bf E}}=\norm{\boldsymbol{\rho}_{h}}\le Ch^{\alpha+k}\left(\norm{{\bf E}}_{1+k}+\norm{P}_{k}\right).
	\end{eqnarray}
	Hence for deriving the optimal ${\bf L}^{2}$-norm error estimate for the electric field, we only have to find a bound the term $\boldsymbol{\chi}_{h}$. 
	\begin{lemma}\label{chithm}
		Under the regularity assumptions,
		\begin{eqnarray*}
			&&{\bf E}\in {\bf H}^{2}\!\left({\bf H}^{1+k}(\Omega)\cap {\bf H}_0(\mathrm{curl};\Omega)\right),\\
			&&P\in H^{2}\!\left(H^{k}(\Omega)\cap H^{1}_{0}(\Omega)\right),
		\end{eqnarray*}
	   	of the true solution of \eqref{model3}, we have the following estimate:
		\begin{equation}\label{chi}
			\norm{\epsilon^{\frac{1}{2}}\boldsymbol{\chi}_{h}}_{{\bf L}^{\infty}\!\left(0,T; {\bf L}^{2}(\Omega)\right)}^{2}
			\le C\,T\,h^{2(\delta+k)}
			\left(
			\int_{0}^{T}\norm{{\bf E}_{ss}}_{1+k}^{2}\,ds
			+
			\int_{0}^{T}\norm{P_{ss}}_{k}^{2}\,ds
			\right).
		\end{equation}
	\end{lemma}
	\begin{proof}
		We initiate the proof by using definition of $\boldsymbol{\chi}_{h}$, then from definition of Ritz projection \eqref{Ritz}, semi-discrete dG algorithm \eqref{sdalg1} and properties of ${\bf T}_{h}$, for any $\boldsymbol{\phi}_{h}\in {\bf V}_{h}^{0}$, we have
		\begin{eqnarray}\label{sL21}
			&&(\epsilon\boldsymbol{\chi}_{htt},\boldsymbol{\phi}_{h})+A(\boldsymbol{\chi}_{h},\boldsymbol{\phi}_{h})-B(\boldsymbol{\phi}_{h},\gamma_{h})\nonumber\\
			&&~~=(\epsilon\left(\mathbb{R}_{h}{\bf E}_{tt}-{\bf E}_{htt}\right),\boldsymbol{\phi}_{h})+A(\mathbb{R}_{h}{\bf E}-{\bf E}_{h},\boldsymbol{\phi}_{h})-B(\boldsymbol{\phi}_{h},\mathbb{E}_{h}P-P_{h})\nonumber\\
			&&~~=(\epsilon\mathbb{R}_{h}{\bf E}_{tt},\boldsymbol{\phi}_{h})+A(\mathbb{R}_{h}{\bf E},\boldsymbol{\phi}_{h})	- B(\boldsymbol{\phi}_{h},\mathbb{E}_{h}P)\nonumber\\
			&&~~~~ - ((\epsilon{\bf E}_{htt},\boldsymbol{\phi}_{h})+A({\bf E}_{h},\boldsymbol{\phi}_{h})	- B(\boldsymbol{\phi}_{h},P_{h}))\nonumber\\
			&&~~=(\epsilon\mathbb{R}_{h}{\bf E}_{tt},\boldsymbol{\phi}_{h})+({\bf f}_{{\bf E},P}, \boldsymbol{\phi}_h)	- ({\bf F},\boldsymbol{\phi}_h) \nonumber\\
			&&~~=(\epsilon\left(\mathbb{R}_{h}{\bf E}_{tt}-{\bf E}_{tt}\right),\boldsymbol{\phi}_h)\nonumber\\
			&&~~=(\epsilon\left(\mathbb{R}_{h}{\bf E}_{tt}-{\bf T}_{h}{\bf E}_{tt}\right),\boldsymbol{\phi}_h)=-(\epsilon\boldsymbol{\rho}_{htt},\boldsymbol{\phi}_h).
		\end{eqnarray}
		Using definition of Ritz projection \eqref{Ritz}, for any $\psi_{h}\in Q_{h}^{0}$, we can deduce
		\begin{eqnarray}\label{sL22}
			B(\boldsymbol{\chi}_{h},\psi_{h})
			&=&B(\mathbb{R}_{h}{\bf E}-{\bf E}_{h},\psi_{h})\nonumber\\
			&=&B(\mathbb{R}_h {\bf E},\psi_{h})-B({\bf E}_{h},\psi_{h})\nonumber\\
			&=&-(g_{{\bf E},P},\psi_{h})+(\rho,\psi_{h})=0.
		\end{eqnarray}
		Differentiating \eqref{sL22} with respect to $t$, then substituting $\psi_{h}=\gamma_{h}$ and adding it to the equation obtained after replacing $\boldsymbol{\phi}_{h}$ by $\boldsymbol{\chi}_{ht}$ in \eqref{sL21}, we get
		\begin{eqnarray}\label{sL23}
			\frac{1}{2}\frac{d}{dt}\left(\norm{\epsilon^{\frac{1}{2}}\boldsymbol{\chi}_{ht}}^{2}+\tnorm{\boldsymbol{\chi}_{h}}^{2}\right)=-(\epsilon\boldsymbol{\rho}_{htt},\boldsymbol{\chi}_{ht}).
		\end{eqnarray}
		Integrating the above equation from $0$ to $t$, then from the fact $\boldsymbol{\chi}_{ht}(\cdot,0)=\boldsymbol{\chi}_{h}(\cdot,0)=0$, Cauchy-Schwarz inequality and estimate \eqref{rho}, leads to
		\begin{eqnarray*}
			&&\norm{\epsilon^{\frac{1}{2}}\boldsymbol{\chi}_{ht}(\cdot,t)}^{2}+\tnorm{\boldsymbol{\chi}_{h}(\cdot,t)}^{2}\nonumber\\
			&&~~= 	\norm{\epsilon^{\frac{1}{2}}\boldsymbol{\chi}_{ht}(\cdot,0)}^{2}+\tnorm{\boldsymbol{\chi}_{h}(\cdot,0)}^{2}-2\int_{0}^{t}(\epsilon\boldsymbol{\rho}_{hss},\boldsymbol{\chi}_{hs})ds\nonumber\\
			&&~~\le2 \Big|\int_{0}^{t}(\epsilon\boldsymbol{\rho}_{hss},\boldsymbol{\chi}_{hs})ds\Big|\nonumber\\
			&&~~\le2 \int_{0}^{t}\norm{\epsilon^{\frac{1}{2}}\boldsymbol{\rho}_{hss}}\norm{\epsilon^{\frac{1}{2}}\boldsymbol{\chi}_{hs}}ds\nonumber\\
			&&~~\le Ch^{\delta+k}\int_{0}^{t}\left(\norm{{\bf E}_{ss}}_{1+k}+\norm{P_{ss}}_{k}\right)\norm{\epsilon^{\frac{1}{2}}\boldsymbol{\chi}_{hs}}ds\nonumber\\
			&&~~\le Ch^{2(\delta+k)}\left(\int_{0}^{t}\norm{{\bf E}_{ss}}_{1+k}^{2}ds+\int_{0}^{t}\norm{P_{ss}}_{k}^{2}ds\right)+\frac{1}{2}\int_{0}^{t}\norm{\epsilon^{\frac{1}{2}}\boldsymbol{\chi}_{hs}}^{2}ds.
		\end{eqnarray*}
		Applying Gronwall's inequality 1 (cf. Lemma \ref{g1}) in above bound, we have
		\begin{equation*}
			\norm{\epsilon^{\frac{1}{2}}\boldsymbol{\chi}_{ht}(\cdot,t)}^{2}+\tnorm{\boldsymbol{\chi}_{h}(\cdot,t)}^{2}\le  Ch^{2(\delta+k)}\left(\int_{0}^{t}\norm{{\bf E}_{ss}}_{1+k}^{2}ds+\int_{0}^{t}\norm{P_{ss}}_{k}^{2}ds\right).~~
		\end{equation*}
		Taking essential supremum over all $t\in(0,T]$ on both sides of above inequality, we have
		\begin{equation}\label{sL25}
			\norm{\epsilon^{\frac{1}{2}}\boldsymbol{\chi}_{ht}}_{{\bf L}^{\infty}({\bf L}^{2}(\Omega))}^{2}+\tnorm{\boldsymbol{\chi}_{h}}_{{\bf L}^{\infty}({\bf L}^{2}(\Omega))}^{2}\le  Ch^{2(\delta+k)}\left(\int_{0}^{T}\norm{{\bf E}_{ss}}_{1+k}^{2}ds+\int_{0}^{T}\norm{P_{ss}}_{k}^{2}ds\right).~~
		\end{equation}
		Further, since $\boldsymbol{\chi}_{h}(\cdot,0)=0$, from \eqref{sL25} we can write
		\begin{eqnarray*}
			\norm{\epsilon^{\frac{1}{2}}\boldsymbol{\chi}_{h}(\cdot,t)}^{2}&=&\norm{\int_{0}^{t}\epsilon^{\frac{1}{2}}\boldsymbol{\chi}_{hs}(\cdot,s)ds+\epsilon^{\frac{1}{2}}\boldsymbol{\chi}_{h}(\cdot,0)}^{2}\nonumber\\
			&\le& CTh^{2(\delta+k)}\left(\int_{0}^{T}\norm{{\bf E}_{ss}}_{1+k}^{2}ds+\int_{0}^{T}\norm{P_{ss}}_{k}^{2}ds\right).
		\end{eqnarray*} 
		This completes the proof.
	\end{proof}
	\begin{thm}\label{optL2}
		Assume the conditions of Lemma \ref{chithm} hold true. Then, we have the following optimal error estimate:
		\begin{eqnarray}\label{SemOptL2}
			\norm{{\bf e}_{h}}_{{\bf L}^{\infty}({\bf L}^{2}(\Omega))}^{2}
			&\le& Ch^{2(\delta+k)}\left(\norm{{\bf E}}^{2}_{{\bf L}^{\infty}({\bf H}^{1+k}(\Omega))}+\norm{P}^{2}_{L^{\infty}(H^{k}(\Omega))}\right)\nonumber\\
			&&+CTh^{2(\delta+k)}\left(\int_{0}^{T}\norm{{\bf E}_{ss}}_{1+k}^{2}ds+\int_{0}^{T}\norm{P_{ss}}_{k}^{2}ds\right).~~ 
		\end{eqnarray}
	\end{thm}
	\begin{proof}
		Combining \eqref{rho} and \eqref{chi}, we can determine
		\begin{eqnarray}
			\norm{{\bf e}_{h}}^{2} &=&\norm{\boldsymbol{\rho}_{h}+\boldsymbol{\chi}_{h}}^{2}\nonumber\\
			&\le&C\norm{\boldsymbol{\rho}_{h}}^{2}+C\norm{\boldsymbol{\chi}_{h}}^{2}\nonumber\\
			&\le&Ch^{2(\delta+k)}\left(\norm{{\bf E}}_{1+k}^{2}+\norm{P}_{k}^{2}\right)\nonumber\\
			&&+CTh^{2(\delta+k)}\left(\int_{0}^{T}\norm{{\bf E}_{ss}}_{1+k}^{2}ds+\int_{0}^{T}\norm{P_{ss}}_{k}^{2}ds\right).~~~~~~~~
		\end{eqnarray}
		Taking the essential supremum of above inequality from $t\in (0,T]$, we get the desired optimal ${\bf L}^{\infty}({\bf L}^{2}(\Omega))$-norm error estimate \eqref{SemOptL2}.
	\end{proof}
	\section{\normalsize Convergence analysis under lower regularity assumptions}
In the previous section, we discussed the convergence of the electric field error under the assumption that \({\bf E}(t) \in {\bf H}^{1+k}(\Omega)\), where the domain \(\Omega\) is convex and is discretized using polygonal or polyhedral elements. In this section, we will examine the convergence of the numerical solution constructed through the scheme \eqref{sdalg1}-\eqref{sdalg4} specifically in cases where \(\Omega\) may not be convex.

For every domain \(\Omega\), there exists a regularity exponent \(\theta \in \left(\frac{1}{2}, 1\right]\) (with \(\theta = 1\) for convex domains). We assume that the solution of equation \eqref{model3} posed in \(\Omega\) satisfies \({\bf E}(t) \in {\bf H}^{\chi+\theta}(\Omega)\), where \(\chi > \frac{1}{2}\) (see Corollary 3.6 of Houston et al. \cite{houston2005interior}). We denote \(\chi + \theta = 1 + s\) for some \(s > 0\). Utilizing the inf-sup condition provided in Lemma A.1 of \cite{duttamohapatra2025} (which holds only for the case when the domain $\Omega$ is partitioned into simplices), we can relax the regularity assumptions on the true solution from Theorems \ref{SemHcurl} and \ref{optL2}, as well as Corollary \ref{fullest}. Under the assumptions \(({\bf E}(t), P(t)) \in \left({\bf H}^{1+s}(\Omega) \cap {\bf H}_{0}(\mathrm{curl};\Omega)\right) \times \left(H^{1}_{0}(\Omega) \cap H^{r}(\Omega)\right)\), with \(s > 0\) and \(1 \leq r \leq k\), we establish the following convergence results which are valid only for simplicial mesh partitions of $\Omega$.
	
		\begin{thm}\label{SemHcurllow}
		Suppose the true solution of model \eqref{model3} regular such that
		\begin{eqnarray*}
			&&{\bf E}\in {\bf H}^{3}({\bf H}^{1+s}(\Omega)\cap {\bf H}_0(\mathrm{curl};\Omega)),\\
			&&P\in H^{2}(H^{r}(\Omega)\cap H^{1}_{0}(\Omega)).
		\end{eqnarray*}
		Then there exists a constant, independent of mesh size $h$, such that we have the following error estimates
		\begin{eqnarray*}
			\mathrm{ess}\sup_{t\in[0,T]}\tnorm{{\bf e}_{h}(t)}_{1}^{2}
			&\le&	Ch^{2\min{\{s,r\}}}\left(\norm{{\bf E}}_{{\bf H}^{3}({\bf H}^{1+s}(\Omega))}^{2}+\norm{P}_{H^{2}(H^{r}(\Omega))}^{2}\right),\nonumber\\
			\mathrm{ess}\sup_{t\in[0,T]}\norm{\tilde{e}_{0}(t)}^{2}
			&	\le&	Ch^{2\min{\{s,r\}}}\left(\norm{{\bf E}}_{{\bf H}^{3}({\bf H}^{1+s}(\Omega))}^{2}+\norm{P}_{H^{2}(H^{r}(\Omega))}^{2}\right).\nonumber\\
		\end{eqnarray*}
		Here $s>0$ and $1\le r\le k$.
	\end{thm}
		\begin{thm}\label{optL2low}
		Let the true solution of model~\eqref{model3} be such regular such that
		\begin{eqnarray*}
			&&{\bf E}\in {\bf H}^{2}({\bf H}^{1+s}(\Omega)\cap {\bf H}_0(\mathrm{curl};\Omega)),\\
			&&P\in H^{2}(H^{r}(\Omega)\cap H^{1}_{0}(\Omega)).
		\end{eqnarray*}
		Then there exists a constant independent of mesh size $h$ such that we have the following error estimate
		\begin{eqnarray*}
			\norm{{\bf e}_{0}}_{{\bf L}^{\infty}({\bf L}^{2}(\Omega))}^{2}
			&\le& Ch^{2(\delta+\min{\{s,r\}})}\left(\norm{{\bf E}}^{2}_{{\bf L}^{\infty}({\bf H}^{1+s}(\Omega))}+\norm{P}^{2}_{L^{\infty}(H^{r}(\Omega))}\right)\nonumber\\
			&&+CTh^{2(\delta+\min{\{s,r\}})}\left(\int_{0}^{T}\norm{{\bf E}_{tt}}_{1+s}^{2}dt+\int_{0}^{T}\norm{P_{tt}}_{r}^{2}dt\right).~~~~~~~~
		\end{eqnarray*}
		Here $0<\delta\le 1$ is the Sobolev regularity exponent of solution of the associated dual problem of \eqref{dual}, $\delta>0$ and $1\le r\le k$.
	\end{thm}
	\begin{rem}
		The proofs of Theorems \ref{SemHcurllow} and \ref{optL2low} follow analogously as the proofs of Theorems \ref{SemHcurl} and \ref{optL2} by appropriately tweaking the approximation properties (cf. Lemma \ref{approximation}) and using the discrete inf-sup condition Lemma A.1 of \cite{duttamohapatra2025}, which holds only for simplicial partitions of domain $\Omega$, in place of Lemma \ref{infsup}.
	\end{rem}
	\begin{rem}
		Consider the following Maxwell model (cf. \cite{qi2025decoupled}): 
		\begin{eqnarray}\label{model4}
			\left\{
			\begin{array}{ll}
				\epsilon{\bf E}_{tt}+\sigma{\bf E}_{t}+\nabla\times (\alpha\nabla\times {\bf E})-\nabla P={\bf F}\; \mbox{in}~\Omega\times (0,T), \\
				\nabla\cdot(\epsilon {\bf E}) = \rho\; \mbox{in}~\Omega\times (0,T), \\
				{\bf E}({\bf x},0)={\bf u}({\bf x}),\,\,
				{\bf E}_t({\bf x},0)={\bf v}({\bf x})\;\;\mbox{on}\; \Omega,\\
				\boldsymbol{\eta}\times{\bf E}=0,\;P=0\;\mbox{on} ~\partial \Omega\times (0,T).
			\end{array}
			\right.
		\end{eqnarray}
		Here, the coefficient $\sigma$ is a non-negative function of position known as the conductivity of the medium. Note that for $\sigma=0$, models \eqref{model3} and \eqref{model4} are the same.
		
		\noindent\rule{\textwidth}{0.4pt}
		\noindent{\bf The semi-discrete parameter-free dG algorithm 2:}
		For $t>0,$  find $({\bf E}_{h}(t),P_{h}(t))\in{\bf V}_{h}^{0}\times Q_{h}^{0}$ satisfying
		\begin{eqnarray}
			&&(\epsilon{\bf E}_{htt},{\bf v}_h)+(\sigma{\bf E}_{ht},{\bf v}_h)+A({\bf E}_{h},{\bf v}_h)-B({\bf v}_{h},P_h)=({\bf F},{\bf v}_h)\;\forall{\bf v}_h\in{\bf V}_h^0,~~~~~~~\label{sdalg11}\\
			&&B({\bf E}_{h},q_h)=-(\rho,q_{h})\; \forall q_h\in Q_{h}^{0},\label{sdalg21}\\
			&&{\bf E}_h({\bf x},0)=\mathbb{R}_h{\bf u}({\bf x}),\,{\bf x}\in\Omega,\label{sdalg31}\\
			&&{\bf E}_{ht}({\bf x},0)=\mathbb{R}_h{\bf v}({\bf x}),\,{\bf x}\in\Omega.\label{sdalg41}
		\end{eqnarray}	
		\noindent\rule{\textwidth}{0.4pt}
		The error estimates in Theorems \ref{SemHcurl}, \ref{SemHcurllow}, \ref{optL2} and \ref{optL2low} presented for the model \eqref{model3} and parameter-free dG algorithm 1 will hold similarly for the model \eqref{model4} and parameter-free dG algorithm 2. The numerical validation of the above semi-discrete problem is presented in Examples \ref{Ex5}, \ref{Ex3}.
	\end{rem}
	\section{\normalsize Numerical computations}\label{sec5}
	
	In this section, we introduce implicit numerical schemes for the model described in \eqref{model4} (or model \eqref{model3} when \(\sigma=0\)). We partition the interval \([0, T]\) into sub intervals \([t_{n-1}, t_{n}]\), where \(t_{n} = n\tau\) for \(n \in \{0, 1, 2, \ldots, N\}\) and \(\tau = \frac{T}{N}\). For a sequence \(\{\xi^{n}\} \subset {\bf L}^{2}(\Omega)\) or \(\{\xi^{n}\} \subset L^{2}(\Omega)\), we define the following difference quotients and operators.
	\begin{eqnarray*}
		\left\{
		\begin{array}{ll}
			\delta_{\tau}\xi^{n}=\frac{\xi^{n+1}-\xi^{n}}{\tau},\,\partial_{\tau}\xi^{n}=\frac{\xi^{n+1}-\xi^{n-1}}{2\tau},\, 	\partial_{\tau}^{2}\xi^{n}=\frac{\xi^{n+1}-2\xi^{n}+\xi^{n-1}}{\tau^{2}},\,\overline{\xi}^{n}=\frac{\xi^{n+1}+\xi^{n}}{2}.
		\end{array}
		\right.
	\end{eqnarray*}
	Furthermore, for a continuous function in the time direction $\beta:[0,T]\to {\bf L}^{2}(\Omega)$ or $\beta:[0,T]\to L^{2}(\Omega)$, we define $\beta^{n}:=\beta(\cdot,t_{n})$, where $n\in\{0,1,2,\cdots,N\}$. Following the notations described above, we now describe two implicit complete discrete schemes as follows:

	\noindent\rule{\textwidth}{0.4pt}
	\noindent{\bf The complete discrete parameter-free dG scheme 1 (CPDG 1):}
	
	For $n\in\{1,2,3,\cdots N-1\}$, we seek $(\mathtt{E}_{h}^{n+1},\mathtt{P}_{h}^{n+1})\in{\bf V}_{h}^{0}\times Q_{h}^{0}$ satisfying:
	\begin{eqnarray}
		&&(\epsilon\partial_{\tau}^{2}\mathtt{E}_{h}^{n},{\bf v}_h)+(\sigma\delta_{\tau}\mathtt{E}_{h}^{n},{\bf v}_{h})+A(\mathtt{E}_{h}^{n+1},{\bf v}_h)-B({\bf v}_{h},\mathtt{P}_{h}^{n+1})=({\bf F}^{n+1},{\bf v}_h)\;\forall{\bf v}_h\in{\bf V}_h^0,~~~~~~~~\label{fdalg1}\\
		&&B(\mathtt{E}_{h}^{n+1},q_h)=-(\rho^{n+1},q_{h})\; \forall q_h\in Q_{h}^{0},\label{fdalg2}\\
		&&\mathrm{with}\nonumber\\
		&&\mathtt{E}_{h}^{0}=\mathbb{R}_h{\bf u}(\textbf{x}),\,\mathrm{and}\,\,\mathtt{E}_{h}^{1}=\mathtt{E}_{h}^{0}+\tau\mathbb{R}_{h}{\bf v}({\bf x}).\label{fdalg3}
	\end{eqnarray}
	\noindent\rule{\textwidth}{0.4pt}
	
	\noindent\rule{\textwidth}{0.4pt}
	\noindent{\bf The complete discrete parameter-free dG scheme 2 (CPDG2):}
	
	For $n\in\{1,2,3,\cdots N-1\}$, we seek $(\mathtt{E}_{h}^{n+1},\mathtt{P}_{h}^{n+1})\in{\bf V}_{h}^{0}\times Q_{h}^{0}$ satisfying:
	\begin{eqnarray}
		&&(\epsilon\partial_{\tau}^{2}\mathtt{E}_{h}^{n},{\bf v}_h)+(\sigma\partial_{\tau}\mathtt{E}_{h}^{n},{\bf v}_{h})+A(\overline{\mathtt{E}}_{h}^{n+1},{\bf v}_h)-B({\bf v}_{h},\mathtt{P}_{h}^{n+1})=(\overline{{\bf F}}^{n+1},{\bf v}_h)\;\forall{\bf v}_h\in{\bf V}_h^0,~~~~~~~~\label{fdalg4}\\
		&&B(\overline{\mathtt{E}}_{h}^{n+1},q_h)=-(\overline{\rho}^{n},q_{h})\; \forall q_h\in Q_{h}^{0},\label{fdalg5}\\
		&&\mathrm{with}\nonumber\\
		&&\mathtt{E}_{h}^{0}=\mathbb{R}_h{\bf u}(\textbf{x}),\,\mathrm{and}\,\,\mathtt{E}_{h}^{1}=\mathtt{E}_{h}^{-1}+2\tau\mathbb{R}_{h}{\bf v}({\bf x}).\label{fdalg6}
	\end{eqnarray}
	\noindent\rule{\textwidth}{0.4pt}
	
	Now we present numerical experiments to validate the theoretical estimates. The 2D and 3D domains are uniformly partitioned using triangular and tetrahedral meshes, respectively. The electric field error $\mathtt{E}_{h}^{N}-{\bf E}^{N}$ is assessed at the final time \( t = t_{N} = T \) and is measured in terms of discrete energy and \( \mathbf{L}^{2} \) norms and the potential error $\mathtt{P}_{h}^{N}-P^{N}$ at the final time-step is measured in the $L^{2}$-norm. The estimated order of convergence is computed using the following formula:
	
	\[
	\text{order} = \frac{\log(e_{i+1}^{N}/e_{i}^{N})}{\log(h_{i+1}/h_{i})}.
	\]
	
	\noindent Here \( e_{j}^{N} \) and \( h_{j} \) (for \( j = i, i+1 \)) represent the error at the final time \( t = T \), measured in appropriate norms, and the mesh size at the \( j^{th} \) iteration, respectively. All the computations in 2D/3D are performed using polynomials of degree $k=1,2$.
	\begin{exm}\label{Ex1}
		\rm{  Consider the Maxwell problem \eqref{model4} in a two dimensional domain $\Omega=(0,1)^{2}$ with the final time $T=1$. The corresponding parameters are selected as $\alpha=\epsilon=1$ and $\sigma=0$.
			Further the exact solution (electric field) is chosen as		 
			\begin{equation*}
				{\bf E}(\textbf{x},t)=t^2\exp(-t)(10x^2(x-1)^2y(y-1)(2y-1),10x(x-1)(2x-1)y^2(y-1)^2).
			\end{equation*}
			This particular choice enforces that $\rho=\nabla\cdot{\bf E}=0$. The potential function $P$ is selected as 
			\begin{equation*}
				P=t^2\exp(t)\sin(2\pi x)\sin(2\pi y).
			\end{equation*} Forcing term and initial data can be extracted from the choice of exact solution.
			\begin{table}[!ht]
				\centering
				\renewcommand{\arraystretch}{1.1}
				\setlength{\tabcolsep}{5pt}
				\begin{tabular}{c c c c c c c c}
					\hline
					$k$ & $h$ & $\norm{{\bf E}^{N} - \mathtt{E                                                                                                                                                                                                                                                                                                                                                                                                                                                                                                                                                                                                                                                                                                                              }_{h}^{N}}$ & Order & 
					$\tnorm{{\bf E}^{N} - \mathtt{E}_{h}^{N}}_{1}$ & Order & 
					$\norm{P^{N}-\mathtt{P}_{h}^{N}}$ & Order \\
					\hline
					& 1/4  & 8.29e-02  & --- & 9.45e-02  & ---   & 6.65e-01 & ---  \\
					1 & 1/8  & 1.84e-02  & 2.17 & 2.75e-02  & 1.78  & 3.53e-01  & 0.91 \\
					& 1/16 & 4.01e-03  & 2.20 & 1.62e-02  & 0.77  & 1.79e-01  & 0.98\\
					& 1/32 & 9.49e-04  & 2.08 & 9.38e-03  & 0.79  & 8.98e-02  & 1.00 \\
					\hline
					& 1/4  & 1.65e-02  & --- & 2.40e-02  & ---  & 2.00e-01  & ---  \\
					2 & 1/8  & 2.93e-03  & 2.49 & 7.25e-03  & 1.73  & 5.33e-02  & 1.91 \\
					& 1/16 & 4.36e-04  & 2.75 & 1.99e-03  & 1.87  & 1.35e-02  & 1.98 \\
					& 1/32 & 5.94e-05  & 2.88 & 5.13e-04  & 1.95 & 3.40e-03 & 1.99\\
					\hline
				\end{tabular}
				\caption{Errors and convergence orders obtained using ${\bf CPDG1}$ scheme measured in different norms for Example \ref{Ex1} with $\tau=h^{\frac{k+1}{2}}$.}
				\label{Tab1}
			\end{table}
			\begin{table}[!ht]
				\centering
				\renewcommand{\arraystretch}{1.1}
				\setlength{\tabcolsep}{5pt}
				\begin{tabular}{c c c c c c c c}
					\hline
					$k$ & $h$ & $\norm{{\bf E}^{N} - \mathtt{E                                                                                                                                                                                                                                                                                                                                                                                                                                                                                                                                                                                                                                                                                                                              }_{h}^{N}}$ & Order & 
					$\tnorm{{\bf E}^{N} - \mathtt{E}_{h}^{N}}_{1}$ & Order & 
					$\norm{P^{N}-\mathtt{P}_{h}^{N}}$ & Order \\
					\hline
					& 1/4  & 8.11e-02  & --- & 9.27e-02  & ---   & 6.65e-01 & --- \\
					1 & 1/8  & 1.84e-02  & 2.14 & 2.75e-02  & 1.75  & 3.53e-01  & 0.91 \\
					& 1/16 & 4.00e-03  & 2.20 & 1.62e-02  & 0.76  & 1.79e-01  & 0.98\\
					& 1/32 & 9.45e-04  & 2.08 & 9.37e-03  & 0.79  & 8.98e-02  & 1.00 \\
					\hline
					& 1/2  & 8.88e-02  & ---  & 1.07e-01  & ---   & 7.01e-01  & ---  \\
					& 1/4  & 1.63e-02  & 2.45 & 2.38e-02  & 2.16  & 2.00e-01  & 1.81  \\
					2 & 1/8  & 2.93e-03  & 2.48 & 7.22e-03  & 1.72  & 5.33e-02  & 1.91 \\
					& 1/16 & 4.33e-04  & 2.76 & 1.96e-03  & 1.88  & 1.35e-02  & 1.98 \\
					\hline
				\end{tabular}
				\caption{Errors and convergence orders obtained using ${\bf CPDG1}$ scheme measured in different norms for Example \ref{Ex1} with $\tau=h^{k+1}$.}
				\label{Tab2}
			\end{table}
			\begin{table}[!ht]
				\centering
				\renewcommand{\arraystretch}{1.1}
				\setlength{\tabcolsep}{5pt}
				\begin{tabular}{c c c c c c c c}
					\hline
					$k$ & $h$ & $\norm{{\bf E}^{N} - \mathtt{E                                                                                                                                                                                                                                                                                                                                                                                                                                                                                                                                                                                                                                                                                                                              }_{h}^{N}}$ & Order & 
					$\tnorm{{\bf E}^{N} - \mathtt{E}_{h}^{N}}_{1}$ & Order & 
					$\norm{P^{N}-\mathtt{P}_{h}^{N}}$ & Order \\
					\hline
					& 1/4  & 8.17e-02  & --- & 9.34e-02  & ---  & 7.53e-01 & ---  \\
					1 & 1/8  & 1.84e-02  & 2.15 & 2.75e-02  & 1.76  & 4.12e-01  & 0.87 \\
					& 1/16 & 4.01e-03  & 2.20 & 1.62e-02  & 0.76  & 2.14e-01  & 0.95\\
					& 1/32 & 9.47e-04  & 2.08 & 9.37e-03  & 0.79  & 1.09e-01  & 0.98 \\
					\hline
					& 1/4  & 1.64e-02  & --- & 2.40e-02  & ---  & 4.31e-01  & --- \\
					2 & 1/8  & 2.93e-03  & 2.48 & 7.23e-03  & 1.73  & 1.30e-01  & 1.72 \\
					& 1/16 & 4.34e-04  & 2.75 & 1.97e-03  & 1.88  & 3.41e-02  & 1.93 \\
					& 1/32 & 5.79e-05  & 2.91 & 5.06e-04  & 1.96 & 8.63e-03  & 1.98\\
					\hline
				\end{tabular}
				\caption{Errors and convergence orders obtained using ${\bf CPDG2}$ scheme measured in different norms for Example \ref{Ex1} with $\tau=h^{\frac{k+1}{2}}$.}
				\label{Tab3}
			\end{table}
			\begin{figure}[!ht]
				\begin{subfigure}{.5\textwidth}
					\centering
					\includegraphics[width=6.00cm, height=6cm]{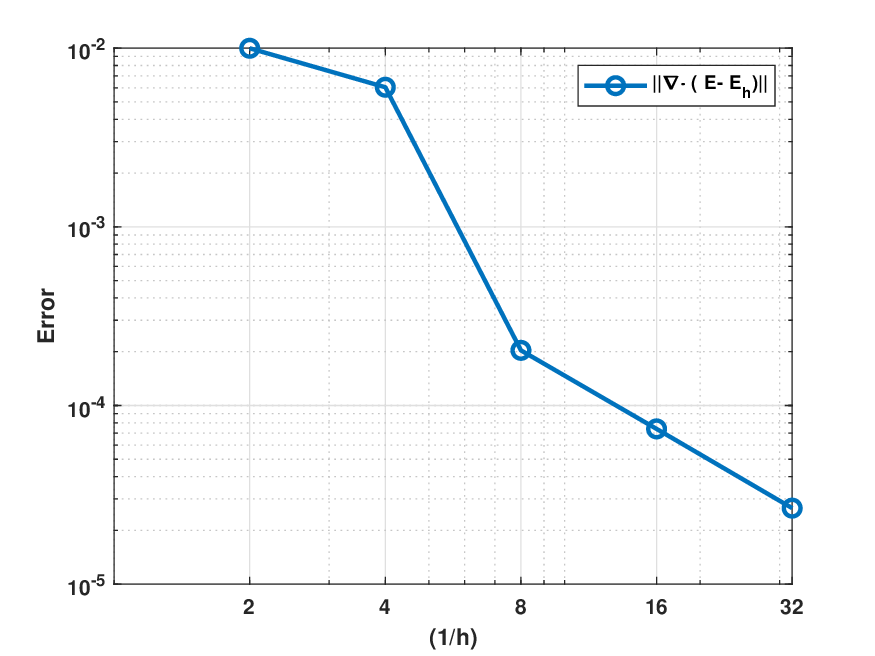}  	  		
				\end{subfigure}
				\begin{subfigure}{.5\textwidth}
					\centering
					\includegraphics[width=6.00cm, height=6cm]{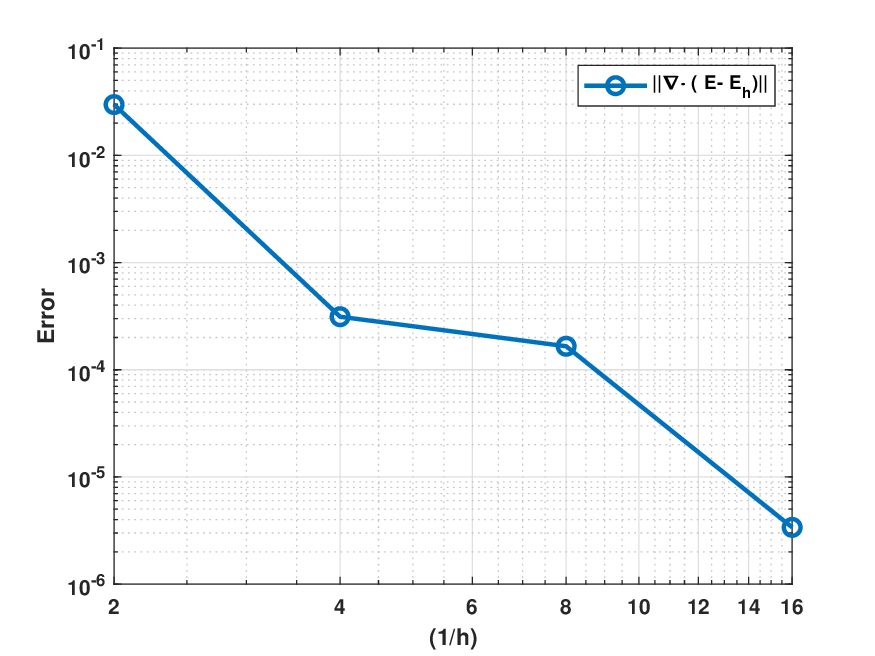}  	
				\end{subfigure}
				\caption{Log-log plots of divergence errors in $L^{2}$-norm (vs.) $(1/h)$, for $k=1$ (left) and $k=2$ (right) for Example \ref{Ex1} computed using {\bf CPDG1} scheme for the case $\tau=h^{k+1}$.}
				\label{Fig1}
			\end{figure}
			
			We now make a few observations from the computational results present in Tables \ref{Tab1}-\ref{Tab3} and Figure \ref{Fig1}.
			\begin{itemize}
				\item From the Tables \ref{Tab1} and \ref{Tab2}, it can be observed that the electric field converges at a rate of $\mathcal{O}(h^{k+1})$ in the ${\bf L}^{2}$ norm and at a rate of $\mathcal{O}(h^{k})$ in discrete energy norms when employing the ${\bf CPDG1}$ (backward Euler) scheme with time steps $\tau=h^{\frac{k+1}{2}}$ and $\tau=h^{k}$, respectively. Furthermore, the potential errors also converge at a rate of $\mathcal{O}(h^{k})$ in $L^{2}$-norm using the aforementioned temporal step sizes.
				
				\item The ${\bf CPDG1}$ scheme has the initial condition $\mathtt{E}_{h}^{1}$ determined using a backward Euler scheme, which is well established in literature to be of a first-order numerical scheme in time. However, we observe from Table \ref{Tab1} that for this example with $\rho=0$ we are getting optimal convergence of the electric field error in ${\bf L}^{2}$-norm even by selecting $\tau=h^{\frac{k+1}{2}}$.
				
				\item Additionally, it is well expected that the ${\bf CPDG2}$ is a Crank-Nicolson type scheme in time and hence is second-order accurate in temporal direction and this can be observed from Table \ref{Tab3} where the time-step $\tau=h^{\frac{k+1}{2}}$ is sufficient for obtaining optimal convergence of the electric field error in ${\bf L}^{2}$-norms. 
				
				\item In Figure \ref{Fig1}, we portray the decrease of divergence errors computed at the final time $t=t_{N}$ using the formula
				\begin{equation*}\norm{\nabla\cdot({\bf E}^{N}-\mathtt{E}_{h}^{N})}^{2}=\sum_{K\in\mathcal{K}_{h}}\norm{\nabla\cdot({\bf E}^{N}-\mathtt{E}_{h}^{N})}_{K}^{2},\end{equation*}
				with respect to the increase in $1/h$ computed using polynomials $k=1,2$. Although the chosen electric field ${\bf E}^{N}$ is exactly divergence-free, the discretely computed electric field $\mathtt{E}_{h}^N$ is not globally divergence-free. To achieve a globally divergence-free solution, one can utilize locally divergence-free polynomial spaces and subsequently project the corresponding dG solution onto its globally divergence-free subspace, as suggested by Cockburn et al. \cite{cockburn2004locally}.
			\end{itemize}
			
		}
	\end{exm}
	\begin{exm}\label{Ex2}
		\rm{On similar lines to Example \ref{Ex1}, we consider the same Maxwell problem but with a different exact solution given as:
			\begin{equation*}
				{\bf E}(\textbf{x},t)=\exp(t)(\sin(\pi x)\sin(\pi y),\sin(\pi x)\sin(\pi y)),\,P=\exp(t)xy(1-x)(1-y).
			\end{equation*}
			This particular choice of ${\bf E}$ enforces that $\rho=\nabla\cdot{\bf E}\ne 0$.
			\begin{table}[!ht]
				\centering
				\renewcommand{\arraystretch}{1.1}
				\setlength{\tabcolsep}{5pt}
				\begin{tabular}{c c c c c c c c}
					\hline
					$k$ & $h$ & $\norm{{\bf E}^{N} - \mathtt{E                                                                                                                                                                                                                                                                                                                                                                                                                                                                                                                                                                                                                                                                                                                              }_{h}^{N}}$ & Order & 
					$\tnorm{{\bf E}^{N} - \mathtt{E}_{h}^{N}}_{1}$ & Order & 
					$\norm{P^{N}-\mathtt{P}_{h}^{N}}$ & Order \\
					\hline
					& 1/2  & 9.80e-01      & ---  & 2.45e+00 & ---   & 7.42e-02       & ---  \\
					& 1/4  & 2.14e-01  & 2.19 & 1.22e+00  & 1.01   & 8.89e-02 &  -0.26 \\
					1 & 1/8  & 4.33e-02  & 2.30 & 6.47e-01  & 0.92  & 3.94e-02  & 1.17 \\
					& 1/16 & 1.44e-02  & 1.59 & 3.34e-01  & 0.96  & 1.37e-02  & 1.52 \\
					\hline
					& 1/2  & 1.92e+00  & ---  & 6.34e+00  & ---   & 9.06e-02  & ---  \\
					& 1/4  & 4.01e-02  & 5.58 & 1.58e-01  & 5.33  & 4.75e-02  & 0.93  \\
					2 & 1/8  & 1.26e-02  & 1.67 & 5.05e-02  & 1.64  & 1.36e-02  & 1.81 \\
					& 1/16 & 3.34e-03  & 1.91 & 1.34e-02  & 1.92  & 3.57e-03  & 1.92 \\
					\hline
				\end{tabular}
				\caption{Errors and convergence orders obtained using ${\bf CPDG1}$ scheme measured in different norms for Example \ref{Ex2} with $\tau=h^{\frac{k+1}{2}}$.}
				\label{Tab4}
			\end{table}
			\begin{table}[!ht]
				\centering
				\renewcommand{\arraystretch}{1.1}
				\setlength{\tabcolsep}{5pt}
				\begin{tabular}{c c c c c c c c}
					\hline
					$k$ & $h$ & $\norm{{\bf E}^{N} - \mathtt{E                                                                                                                                                                                                                                                                                                                                                                                                                                                                                                                                                                                                                                                                                                                              }_{h}^{N}}$ & Order & 
					$\tnorm{{\bf E}^{N} - \mathtt{E}_{h}^{N}}_{1}$ & Order & 
					$\norm{P^{N}-\mathtt{P}_{h}^{N}}$ & Order \\
					\hline
					& 1/2  & 1.23e+00      & ---  & 2.53e+00 & ---   & 2.33e-01       & ---  \\
					& 1/4  & 1.99e-01  & 2.63 & 1.22e+00  & 1.05   & 1.01e-01 & 1.20  \\
					1 & 1/8  & 4.03e-02  & 2.30 & 6.44e-01  & 0.92  & 4.89e-02  & 1.05 \\
					& 1/16 & 8.93e-03  & 2.17 & 3.31e-01  & 0.96  & 1.70e-02  & 1.52 \\
					\hline
					& 1/2  & 2.23e-01  & ---  & 5.25e-01  & ---   & 1.43e-01  & ---  \\
					& 1/4  & 2.75e-02  & 3.02 & 1.27e-01  & 2.05  & 4.36e-02  & 1.71  \\
					2 & 1/8  & 2.56e-03  & 3.43 & 3.12e-02  & 2.02  & 1.20e-02  & 1.87 \\
					& 1/16 & 3.04e-04  & 3.07 & 7.82e-03  & 2.00  & 3.13e-03  & 1.94 \\
					\hline
				\end{tabular}
				\caption{Errors and convergence orders obtained using ${\bf CPDG1}$ scheme measured in different norms for Example \ref{Ex2} with $\tau=h^{k+1}$.}
				\label{Tab5}
			\end{table}
			\begin{table}[!ht]
				\centering
				\renewcommand{\arraystretch}{1.1}
				\setlength{\tabcolsep}{5pt}
				\begin{tabular}{c c c c c c c c}
					\hline
					$k$ & $h$ & $\norm{{\bf E}^{N} - \mathtt{E                                                                                                                                                                                                                                                                                                                                                                                                                                                                                                                                                                                                                                                                                                                              }_{h}^{N}}$ & Order & 
					$\tnorm{{\bf E}^{N} - \mathtt{E}_{h}^{N}}_{1}$ & Order & 
					$\norm{P^{N}-\mathtt{P}_{h}^{N}}$ & Order \\
					\hline
					& 1/2  & 1.28e+00      & ---  & 2.64e+00 & ---   & 2.32e-01       & ---  \\
					& 1/4  & 2.05e-01  & 2.64 & 1.32e+00  & 1.01   & 6.77e-02 & 1.78  \\
					1 & 1/8  & 4.20e-02  & 2.29 & 6.86e-01  & 0.94  & 2.97e-02  & 1.19 \\
					& 1/16 & 1.10e-02  & 1.94 & 3.41e-01  & 1.01  & 4.03e-02  & -0.44 \\
					& 1/32 & 3.93e-03  & 1.48 & 1.69e-01  & 1.02  & 4.66e-02  & -0.21 \\
					\hline
					& 1/2  & 2.22e-01  & ---  & 5.63e-01  & ---   & 1.29e-01  & ---  \\
					& 1/4  & 2.74e-02  & 3.02 & 1.29e-01  & 2.12  & 5.04e-02  & 1.36  \\
					2 & 1/8  & 3.32e-03  & 3.05 & 3.20e-02  & 2.02  & 3.02e-02  & 0.74 \\
					& 1/16 & 5.45e-04  & 2.61 & 7.97e-03  & 2.00  & 2.82e-02  & 0.10 \\
					\hline
				\end{tabular}
				\caption{Errors and convergence orders obtained using ${\bf CPDG2}$ scheme measured in different norms for Example \ref{Ex2} with $\tau=h^{\frac{k+1}{2}}$.}
				\label{Tab6}
			\end{table}
			
			\begin{table}[!ht]
				\centering
				\renewcommand{\arraystretch}{1.1}
				\setlength{\tabcolsep}{5pt}
				\begin{tabular}{c c c c c c c c}
					\hline
					$k$ & $h$ & $\norm{{\bf E}^{N} - \mathtt{E                                                                                                                                                                                                                                                                                                                                                                                                                                                                                                                                                                                                                                                                                                                              }_{h}^{N}}$ & Order & 
					$\tnorm{{\bf E}^{N} - \mathtt{E}_{h}^{N}}_{1}$ & Order & 
					$\norm{P^{N}-\mathtt{P}_{h}^{N}}$ & Order \\
					\hline
					& 1/2  & 1.35e+00      & ---  & 2.60e+00 & ---   & 5.65e-01       & ---  \\
					& 1/4  & 1.94e-01  & 2.81 & 1.22e+00  & 1.09   & 8.44e-01 & -0.58  \\
					1 & 1/8  & 4.06e-02  & 2.25 & 6.44e-01  & 0.93  & 3.16e+00  & -1.90 \\
					& 1/16 & 9.00e-03  & 2.17 & 3.31e-01  & 0.96  & 1.24e+01  & -1.98 \\
					& 1/32 & 2.13e-03  & 2.08 & 1.67e-01  & 0.98  & 4.94e+01  & -1.99 \\
					\hline
					& 1/2  & 2.14e-01  & ---  & 5.26e-01  & ---   & 1.49e-01  & ---  \\
					& 1/4  & 2.83e-02  & 2.92 & 1.28e-01  & 2.04  & 3.55e-01  & -1.25  \\
					2 & 1/8  & 2.46e-03  & 3.52 & 3.17e-02  & 2.01  & 1.66e+00  & -2.22 \\
					& 1/16 & 2.99e-04  & 3.04 & 8.05e-03  & 1.98  & 7.04e+00  & -2.09 \\
					\hline
				\end{tabular}
				\caption{Errors and convergence orders obtained using ${\bf CPDG2}$ scheme measured in different norms for Example \ref{Ex2} with $\tau=h^{k+1}$.}
				\label{Tab7}
			\end{table}
			\begin{figure}[!ht]
				\centering
				\includegraphics[width=7.00cm, height=6cm]{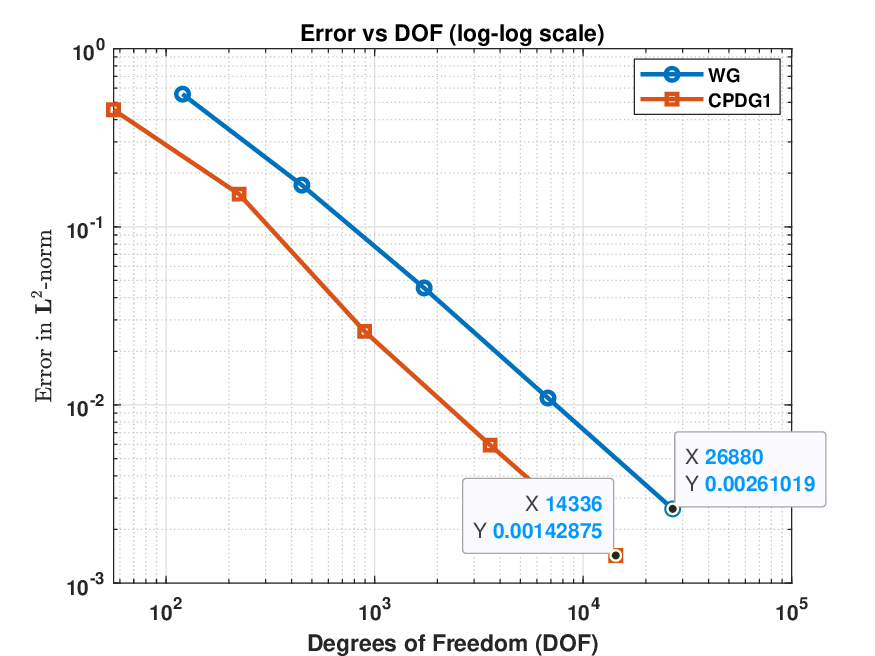}  	  		
				\caption{Log-log plot of electric field error computed in ${\bf L}^{2}$-norms of WG (cf. \cite{qi2025decoupled}) and {\bf CPDG1} vs. the DOFs using linear elements.}
				\label{Fig2}
			\end{figure}
			We summarize the findings of this experiment in Tables \ref{Tab4}-\ref{Tab7} and Figure \ref{Fig2} as follows:
			\begin{itemize}
				\item Unlike the case as in Example \ref{Ex1}, here we have $\rho\ne0$. Hence due to this, in Table \ref{Tab4} by selecting $\tau=h^{\frac{k+1}{2}}$  we observe that the ${\bf CPDG1}$ shows a drop in the order of convergence in the ${\bf L}^{2}$-norm of the error in electric field upon finer mesh refinement.
				
				\item The order is recovered back by selecting the time-step as $\tau=h^{k+1}$ and such phenomenon is observed in Table \ref{Tab5}. This was expected because the ${\bf CPDG1}$ scheme has a first-order initial startup.
				
				\item Additionally, we note that optimal convergence of the errors in the discrete energy norm for the error in electric field and in $L^{2}$-norm for the error in potential function is achieved using the {\bf CPDG1} scheme irrespective of the choice of $\tau$ as $h^{k+1}$ or $h^{\frac{k+1}{2}}$.
				
				\item Although ${\bf CPDG2}$ is a Crank-Nicolson type scheme in the temporal direction, we observe similar convergence phenomena in Tables \ref{Tab6} and \ref{Tab7} as in ${\bf CPDG1}$ (see Tables \ref{Tab4} and \ref{Tab5}). Specifically, it behaves like a first-order scheme in time when we select $\tau = h^{\frac{k+1}{2}}$. The order in the ${\bf L}^{2}$-norm for the error in the electric field can be recovered by choosing $\tau = h^{k+1}$. Radu and Egger noted a similar phenomenon in their work \cite{egger2021second} (refer to numerical test 5.2 in \cite{egger2021second}), where they employed a curl-conforming explicit finite element method with mass lumping to address the non-zero divergence of the electric field. Additionally, a modified $\mathcal{EJ}_{1}^{*}$ Nédélec element (see \cite{egger2021second}) was used to recover the order of convergence.
				
				\item One important drawback of the ${\bf CPDG2}$ scheme is noted in Tables \ref{Tab6} and \ref{Tab7}, where the potential errors diverge and increase with mesh refinement. This phenomenon renders it ineffective for such Maxwell problems when $\rho\ne 0$. 
				
				\item In Figure \ref{Fig2}, we present the log-log plots of ${\bf L}^{2}$-norm error in electric field vs. the degrees of freedom between the proposed ${\bf CPDG1}$ scheme and the backward Euler weak Galerkin (WG) method, as discussed in \cite{qi2025decoupled}. For the ${\bf CPDG1}$ scheme, we utilize piecewise linear polynomials, while the WG method employs the lowest order element specified in \cite{qi2025decoupled}. Both methods achieve a comparable level of accuracy, as evidenced by the errors in the electric field measured in the ${\bf L}^{2}$-norm. However, it is noteworthy that the proposed ${\bf CPDG1}$ scheme requires significantly fewer degrees of freedom (DOF) compared to the WG methods.
				
			\end{itemize}
		}
	\end{exm}
		\begin{exm}\label{Ex5}
		\rm{  Consider the Maxwell problem \eqref{model4} in $\Omega=(0,1)^{2}$ with the final time $T=1$ and $\alpha=\epsilon=\sigma=1$.
			Further the exact solution is chosen as		 
			\begin{equation*}
				{\bf E}=\cos(t)\nabla \phi(r,\theta),\mathrm{with}\, \phi(r,\theta)= r^{1+l}\sin((1+l)\theta),\,0<l\le1.
			\end{equation*}
			The potential function $P$ is selected as 
			\begin{equation*}
				P=\cos(t)\sin(\pi x)\sin(\pi y).
			\end{equation*} 
			Here ${\bf E}(t)\in {\bf H}^{1+\l-\varepsilon}(\Omega)$ where $\varepsilon>0$ is a very small positive number with $l-\varepsilon>0$ and $P$ is smooth.
			\begin{table}[!ht]
				\centering
				\renewcommand{\arraystretch}{1.1}
				\setlength{\tabcolsep}{5pt}
				\begin{tabular}{c c c c c c c c}
					\hline
					$l$ & $h$ & $\norm{{\bf E}^{N} - \mathtt{E                                                                                                                                                                                                                                                                                                                                                                                                                                                                                                                                                                                                                                                                                                                             }_{h}^{N}}$ & Order & 
					$\tnorm{{\bf E}^{N} - \mathtt{E}_{h}^{N}}_{1}$ & Order & 
					$\norm{P^{N}-\mathtt{P}_{h}^{N}}$ & Order \\
					\hline
					& 3.536e-01  & 1.499e-01      & ---  & 3.500e-01 & ---   & 1.634e-01      & ---  \\
					& 1.768e-01  & 7.912e-02  & 0.92 & 2.945e-01  & 0.25   & 1.014e-01 & 0.69  \\
					$\frac{1}{10}$ & 8.839e-02  & 4.349e-02  & 0.86  & 2.568e-01  & 0.20  & 7.691e-02  & 0.40 \\
					& 4.419e-02 & 2.257e-02  & 0.95 & 2.332e-01  & 0.14 & 6.343e-02  & 0.28\\
					& 2.210e-02 & 1.136e-02  & 0.99 & 2.143e-01  & 0.12  & 6.052e-02  & 0.07 \\
					\hline
					& 3.536e-01  & 1.227e-01      & ---  & 1.716e-01 & ---   & 1.507e-01      & ---  \\
					& 1.768e-01  & 1.568e-02   & 2.97 & 7.800e-02  & 1.14   & 7.522e-02 & 1.00  \\
					$\frac{2}{3}$ & 8.839e-02  & 4.454e-03  & 1.82 & 4.041e-02  & 0.95  & 4.402e-02  & 0.77 \\
					& 4.419e-02 & 1.451e-03  & 1.62 & 2.410e-02   & 0.75  & 2.337e-02  & 0.91\\
					& 2.210e-02 & 4.731e-04 & 1.62 & 1.424e-02  & 0.76  & 1.198e-02  & 0.96 \\
					\hline
				\end{tabular}
				\caption{Errors and convergence orders obtained using ${\bf CPDG1}$ scheme measured in different norms with $\tau=h^{2}$ for Example \ref{Ex5}.}
				\label{ch7:Tab2*}
			\end{table}
		}
		
		In Table \ref{ch7:Tab2*}, we present the convergence rates of the errors in electric field and potential function of the ${\bf CPDG1}$ scheme with $\tau=h^{2}$ for the values of $l=\{\frac{1}{10},\frac{2}{3}\}$. As presented in Theorems \ref{SemHcurllow} and \ref{optL2low}, we observe the convergence rates of $\mathcal{O}(h^{1+l})$ in the ${\bf L}^{2}$-norm for the error in electric field and $\mathcal{O}(h^{l})$ in the discrete energy norm for the electric field error and the same rate in the $L^{2}$-norm for the potential error. 
	\end{exm}
	\begin{exm}\label{Ex3}
		\rm{
			This example discusses the convergence of the scheme {\bf CPDG1} for a Maxwell problem in a 3D domain  $\Omega=(0,1)^{3}$ with the parameters $\alpha=\epsilon=\sigma=1$. The exact solution is selected as (cf. \cite{da2022virtual})
			\begin{equation*}
				{\bf E}=t\nabla\times{\bf }\boldsymbol{\phi}+t^2\boldsymbol{\psi},\, P=0,
			\end{equation*}	
			where,
			\begin{eqnarray*}
				\boldsymbol{\phi}(x,y,z)=
				\begin{pmatrix}
					\sin^{2}(\pi x)\, y^{2}(1-y)^{2}\, z^{2}(1-z)^{2} \\[6pt]
					x^{2}(1-x)^{2}\, \sin^{2}(\pi y)\, z^{2}(1-z)^{2} \\[6pt]
					x^{2}(1-x)^{2}\, y^{2}(1-y)^{2}\, \sin^{2}(\pi z)
				\end{pmatrix}, 
				\boldsymbol{\psi}(x,y,z)
				= \nabla\big(\sin(\pi x)\sin(\pi y)\sin(\pi z)\big).			
			\end{eqnarray*}
			\begin{table}[!ht]
				\centering
				\renewcommand{\arraystretch}{1.1}
				\setlength{\tabcolsep}{5pt}
				\begin{tabular}{c c c c c c c c}
					\hline
					$k$ & $h$ & $\norm{{\bf E}^{N} - \mathtt{E                                                                                                                                                                                                                                                                                                                                                                                                                                                                                                                                                                                                                                                                                                                              }_{h}^{N}}$ & Order & 
					$\tnorm{{\bf E}^{N} - \mathtt{E}_{h}^{N}}_{1}$ & Order & 
					$\norm{P^{N}-\mathtt{P}_{h}^{N}}$ & Order \\
					\hline
					& 1/2  & 6.19e-01      & ---  & 7.26e-01 & ---   & 1.15e+00       & ---  \\
					& 1/4  & 1.61e-01  & 1.94 & 4.95e-01  & 0.55   & 3.71e-01 & 1.64  \\
					1 & 1/6  & 7.47e-02  & 1.89 & 4.11e-01  & 0.46  & 2.95e-01  & 0.56 \\
					& 1/8 & 4.19e-02  & 2.01 & 3.40e-01  & 0.66  & 2.22e-01 & 0.99 \\
					& 1/10 & 2.67e-02  & 2.02 & 2.87e-01  & 0.76  &	1.77e-01  & 1.01 \\
					\hline
					& 1/2  & 1.26e-01  & ---  & 2.38e-01  & ---   & 5.04e-01  & ---  \\
					& 1/4  & 1.46e-02  & 3.11 & 8.90e-02  & 1.42  & 4.51e-02  & 3.48  \\
					2 & 1/6  & 4.46e-03  & 2.92 & 4.43e-02
					& 1.72  & 2.01e-02  & 2.00 \\
					& 1/8 & 1.92e-03  & 2.93 & 2.63e-02  & 1.81  & 1.13e-02
					& 2.00 \\
					\hline
				\end{tabular}
				\caption{Errors and convergence orders obtained using ${\bf CPDG1}$ scheme measured in different norms for Example \ref{Ex3} with $\tau=h^{k+1}$.}
				\label{Tab8}
			\end{table}
			We enlist the observations made from Table \ref{Tab8} below
			\begin{itemize}
				\item The errors in the electric field exhibit optimal convergence in both ${\bf L}^{2}$ and discrete energy norms when we set $\tau = h^{k+1}$. Similarly, the errors in the potential function, measured using the $L^{2}$-norm, also demonstrate optimal convergence rates. This confirms our theoretical predictions. The selection of \(\tau\) was influenced by Example \ref{Ex2}, which indicated that when \(\rho \neq 0\), it is essential to set \(\tau = h^{k+1}\) to achieve optimal convergence.
				
				\item In \cite{da2022virtual}, a lowest-order VEM for the time-dependent Maxwell equations in a first-order formulation was discussed. In numerical test case 1 of \cite{da2022virtual}, the convergence of the VEM numerical scheme was demonstrated, where a rate of $\mathcal{O}(h)$ was observed for the error of the electric field in the ${\bf L}^{2}$-norm. However, for the same example, we achieve the rate $\mathcal{O}(h^{2})$ using linear dG elements, which is indeed optimal.
			\end{itemize}
		}
	\end{exm}
	\begin{exm}\label{Ex4}
		\rm{
			In this example, we study the propagation of the numerical electric fields obtained for the Maxwell system \eqref{model4} with the parameters given as $\epsilon=3$, $\sigma=1$ and $\mu=\frac{1}{5}$. The exact solution is unknown here but the forcing term and initial data are supplied to be
			\begin{eqnarray*}
				{\bf u}=(\sin(\pi x)\sin(\pi y),x^2y^2(1-x)^2(1-y)^2),\,{\bf v}={\bf f}=0.
			\end{eqnarray*} 
			Further, we set $\rho=(1+t^2)(\pi\cos(\pi x)\sin(\pi y)+2x^2y(1-x)^2(1-y)^2+2x^2y^2(1-x)^2(1-y))$. It is clear that the initial wave speed is zero. In Figures \ref{Fig3} - \ref{Fig5}, we demonstrate the propagation of electric fields across time levels $T=1,2,10$ by supplementing the contour and surface plots of the electric fields computed using linear elements and $\tau=2h$ with $h=\frac{1}{16}$. It is observed that the amplitude of the electric fields increases with rise in time.
			\begin{figure}[!ht]
				\centering
				\includegraphics[width=6.00cm, height=5cm]{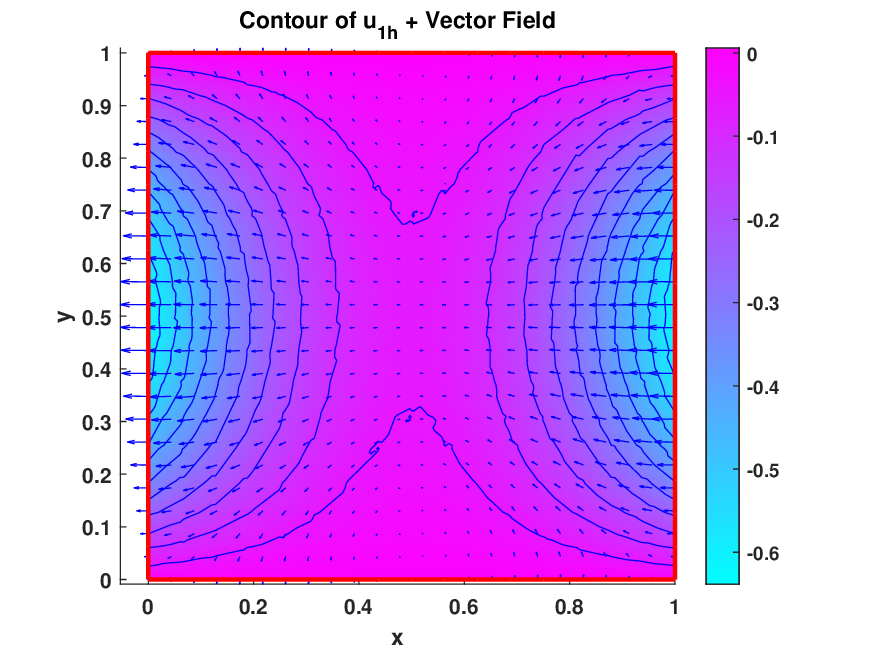}  
				\includegraphics[width=6.00cm, height=5cm]{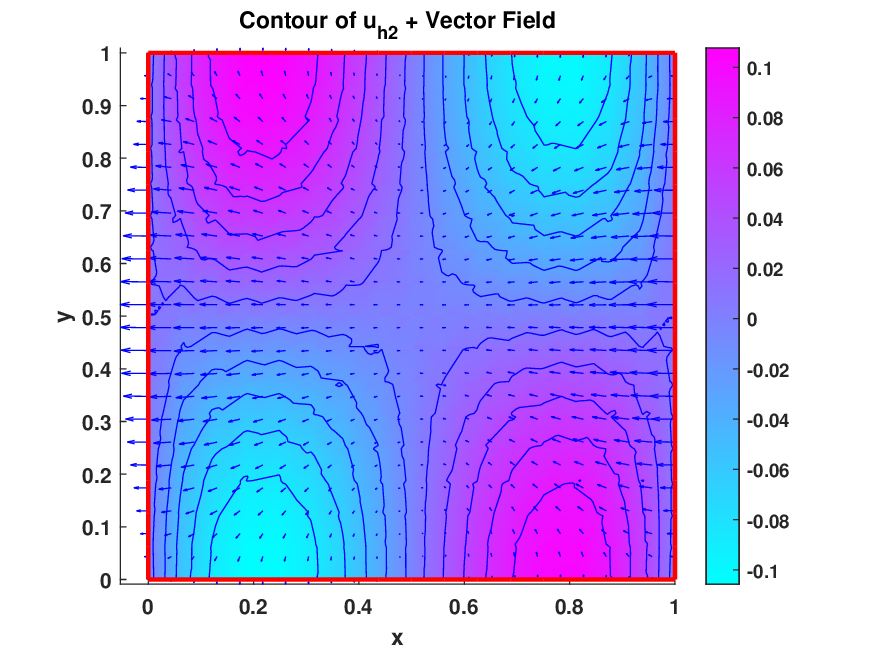}  	  		
				\includegraphics[width=6.00cm, height=5cm]{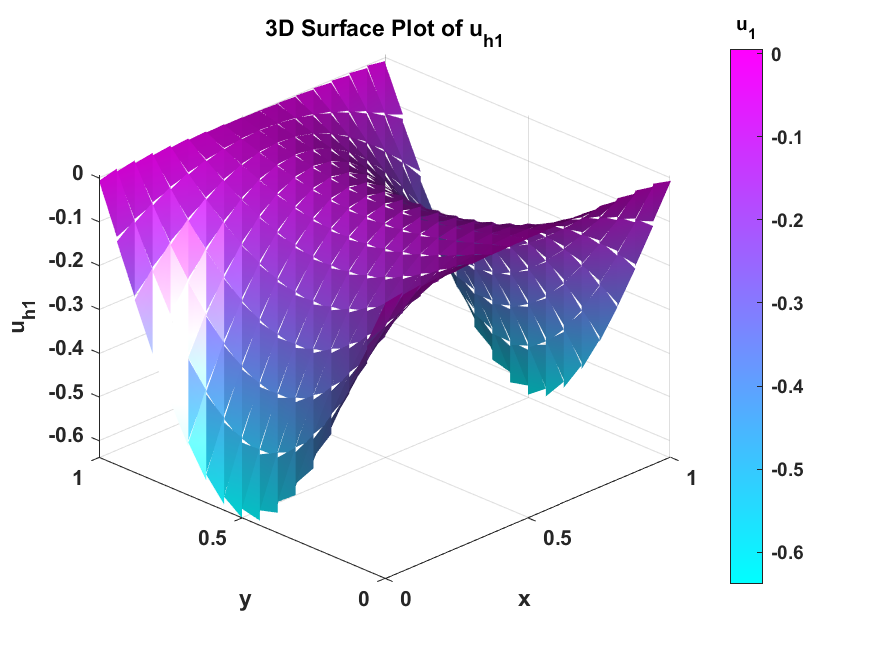}
				\includegraphics[width=6.00cm, height=5cm]{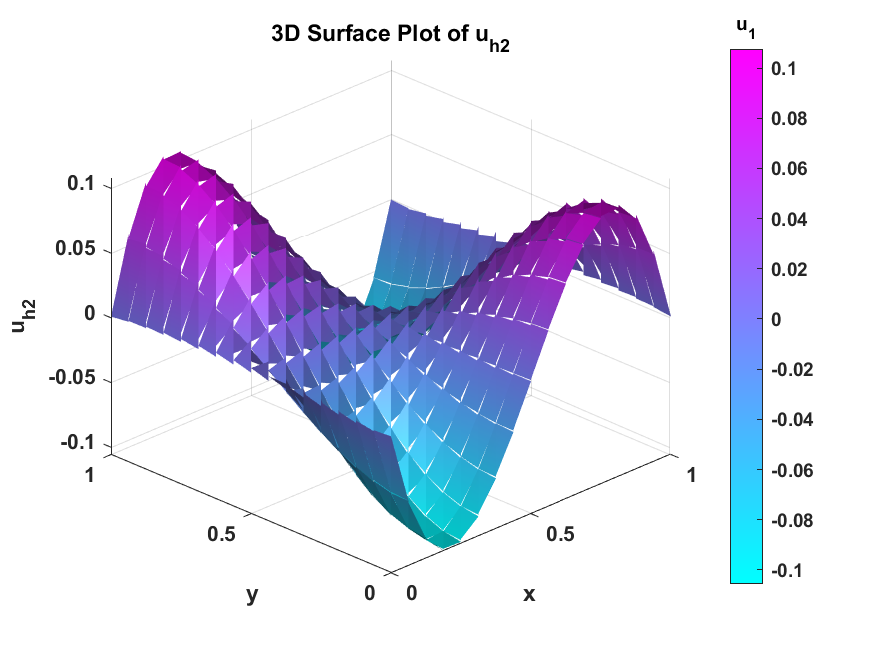}
				\caption{Contour and surface plots of electric field components for Example \ref{Ex4} computed using {\bf CPDG1} scheme with $k=1$, $\tau=2h$ and $T=1$.}
				\label{Fig3}
			\end{figure}	
			
			\begin{figure}[!ht]
				\centering
				\includegraphics[width=6.00cm, height=5cm]{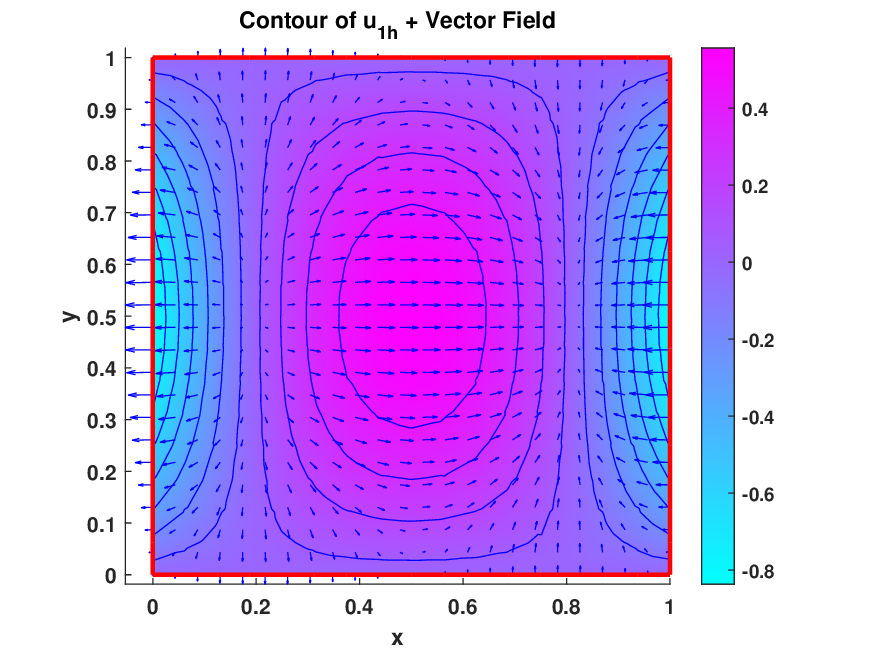}  
				\includegraphics[width=6.00cm, height=5cm]{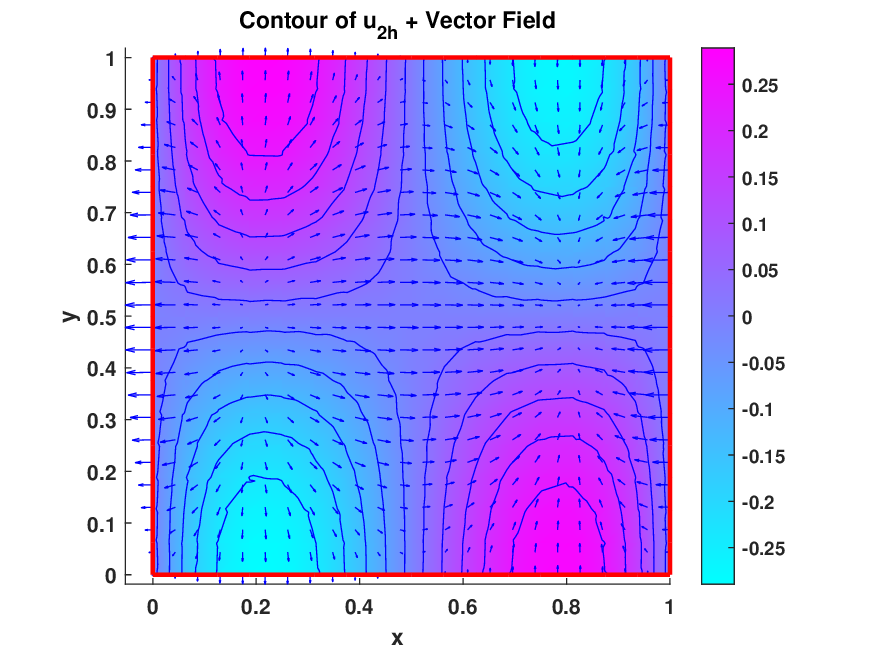}  	  		
				\includegraphics[width=6.00cm, height=5cm]{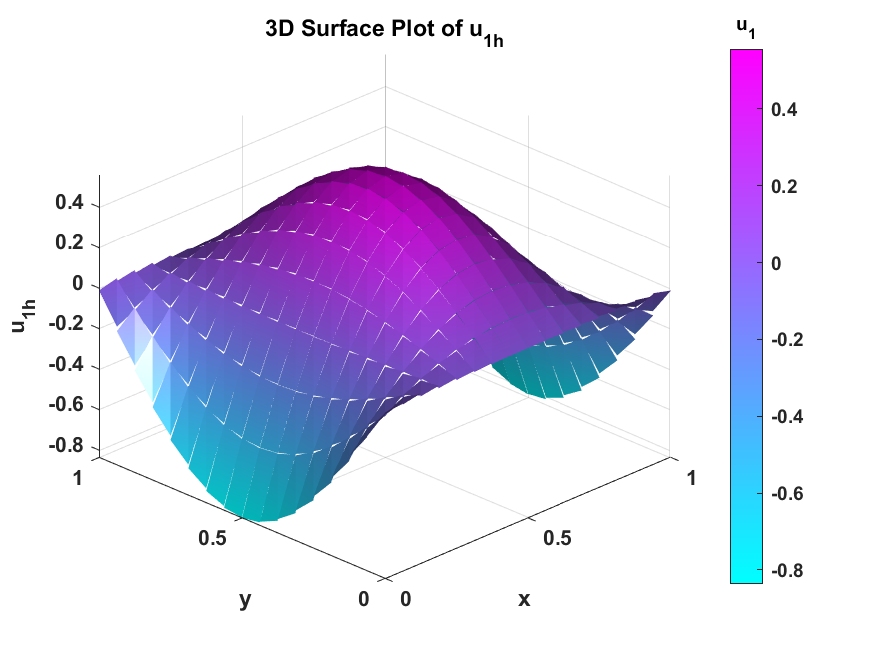}
				\includegraphics[width=6.00cm, height=5cm]{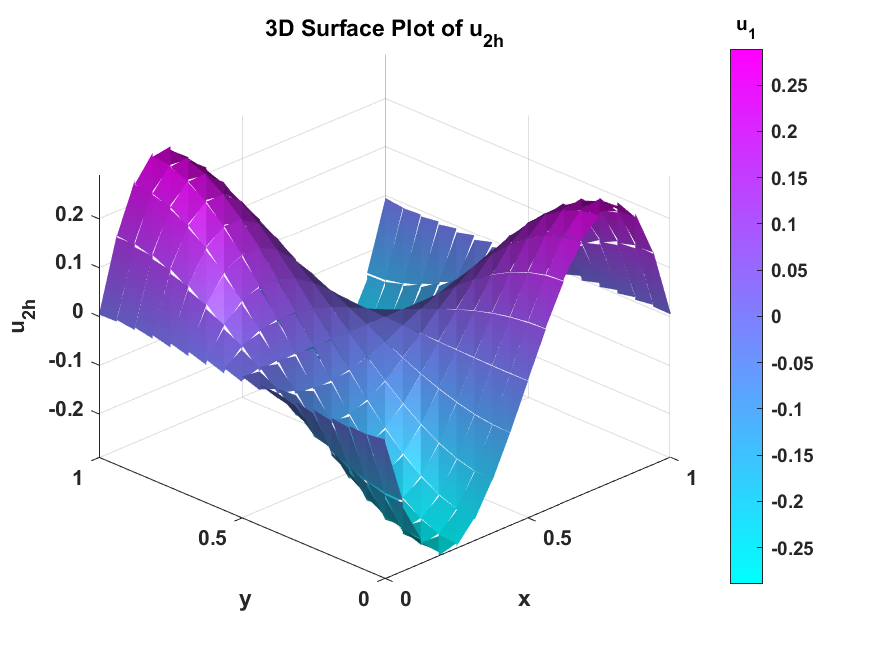}
				\caption{Contour and surface plots of electric field components for Example \ref{Ex4} computed using {\bf CPDG1} scheme with $k=1$, $\tau=2h$ and $T=2$.}
				\label{Fig4}
			\end{figure}	
			\begin{figure}[!ht]
				\centering
				\includegraphics[width=6.00cm, height=5cm]{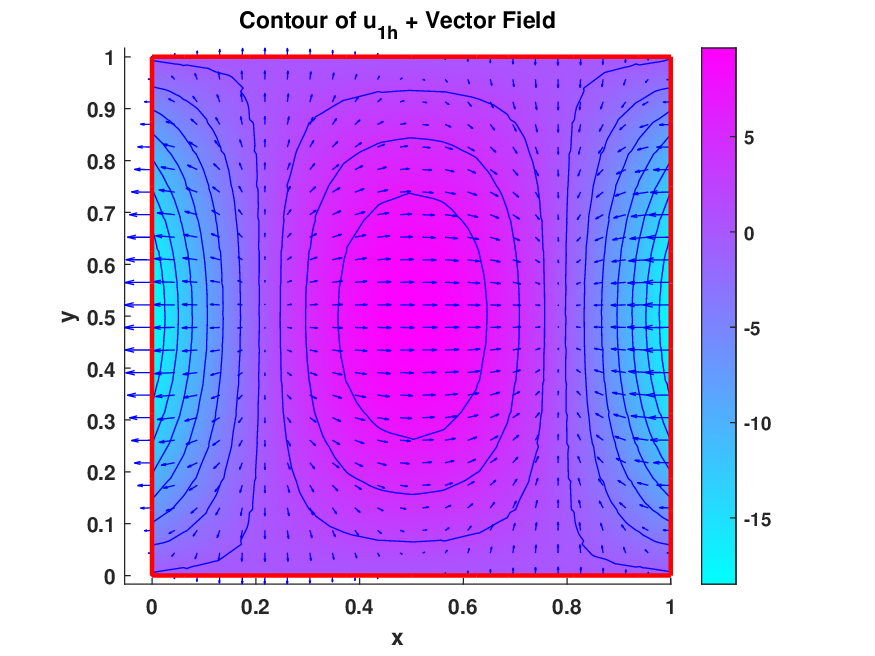}  
				\includegraphics[width=6.00cm, height=5cm]{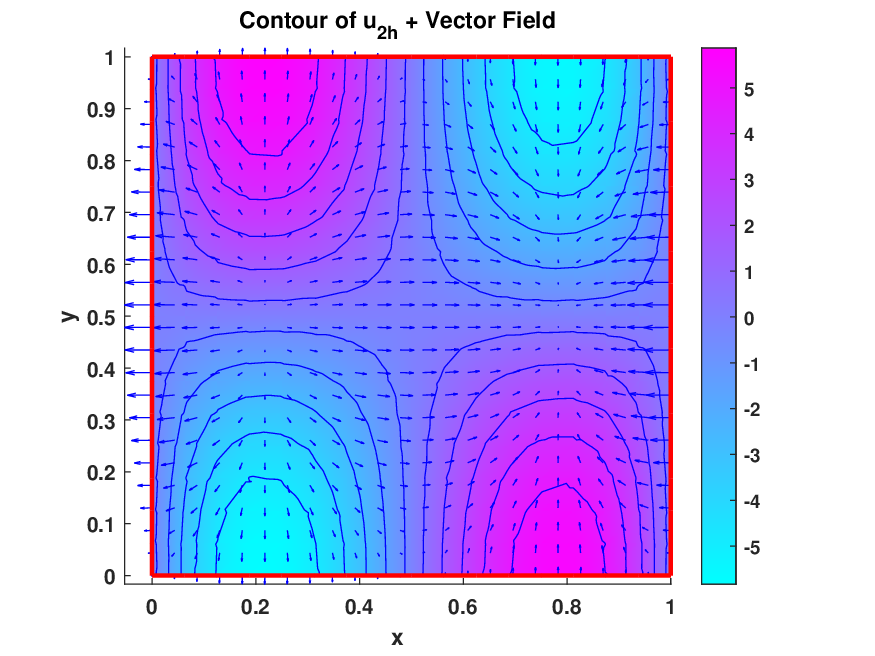}  	  		
				\includegraphics[width=6.00cm, height=5cm]{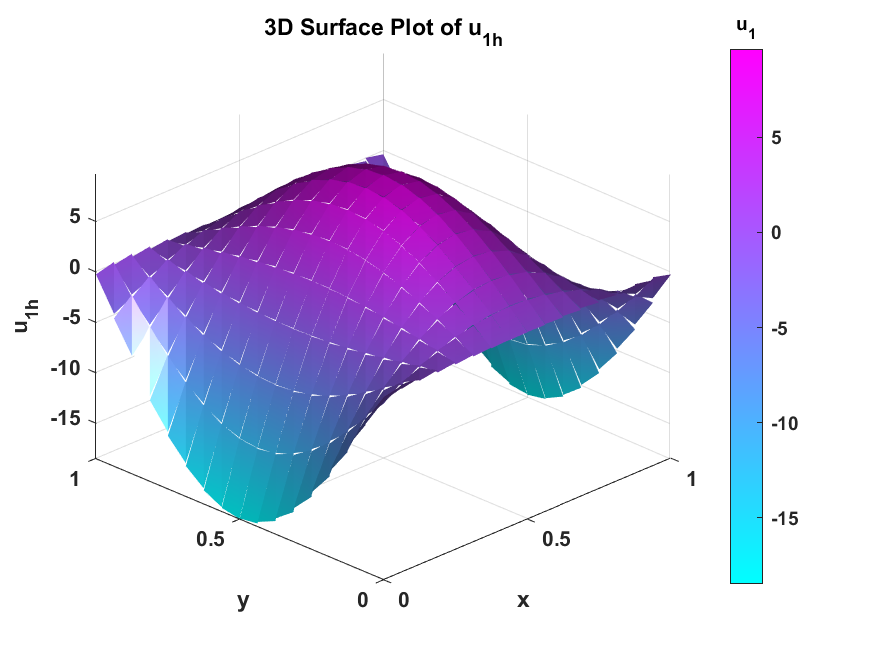}
				\includegraphics[width=6.00cm, height=5cm]{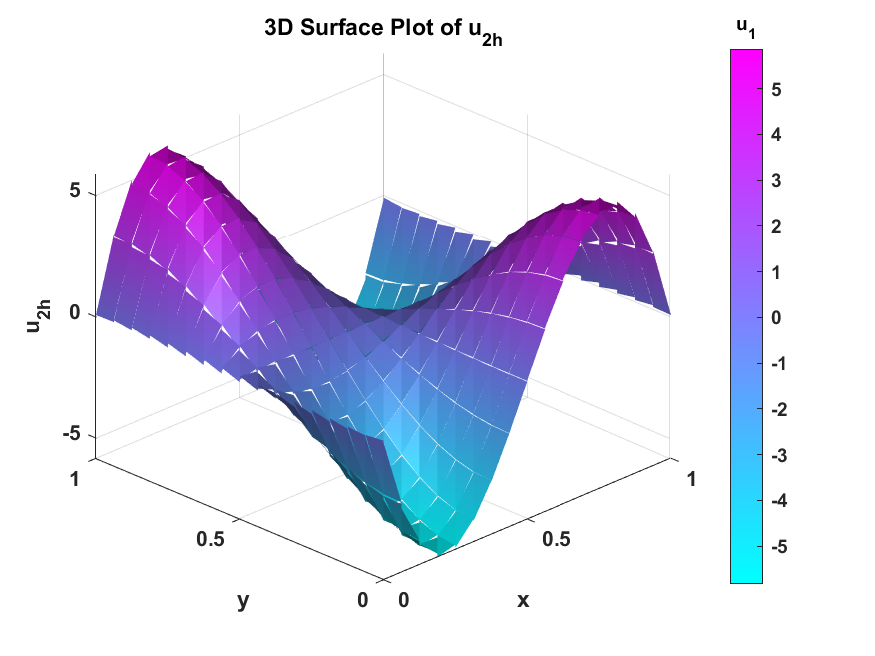}
				\caption{Contour and surface plots of electric field components for Example \ref{Ex4} computed using {\bf CPDG1} scheme with $k=1$, $\tau=2h$ and $T=10$.}
				\label{Fig5}
			\end{figure}	
		}
	\end{exm}
	\section{\normalsize Conclusion}\label{sec6}
	We have proposed a parameter-free discontinuous Galerkin finite element scheme for the second-order time-dependent Maxwell equations. Further, we establish the stability of the continuous-in-time scheme and conduct a rigorous error analysis to achieve optimal convergence rates for the electric field variables in ${\bf L}^{2}$ and discrete energy norms, as well as for the potential function in the $L^{2}$-norm. The main idea for obtaining the ${\bf L}^{2}$-norm error estimate for the electric field hinges on the Ritz (elliptic)  projection pair as introduced in \eqref{Ritz}. Being a saddle point hyperbolic problem, the analysis requires delicate handling by using the discrete inf-sup condition (cf. Lemma \ref{infsup}), unlike the standard approaches of assuming that the charge conservation law holds implicitly as done in \cite{mohapatra2025new,ciarlet1999fully,monk1992analysis}. Two time integration schemes are introduced, and their convergence performance is thoroughly examined through 2D and 3D numerical experiments. It is revealed that the ${\bf CPDG1}$ and ${\bf CPDG2}$ both behave as second-order accurate schemes in temporal direction when the electric fields are divergence-free. However, the accuracy reduces by one for the general case for both the proposed complete discrete schemes and the theoretical analysis of this corresponding phenomena is still open. Also, our findings indicate that the proposed methods are more computationally efficient for Maxwell problems compared to the recently developed WG methods in \cite{qi2025decoupled}.
	
		\section*{\normalsize Funding} The Ministry of Education, Government of India, provided funding for the first
		author under the Prime Minister’s Research Fellowship (PMRF) (Ref: PMRF
		ID-1902166). The second author gratefully acknowledges the research support
		of the Science and Engineering Research Board (SERB), Govt. of India, through
		the Grant vide CRG Project no. CRG/2023/006232.
	
	\section*{\normalsize Data Availability}
	There is no data associated with this manuscript.
	
	\section*{\normalsize Declarations}
	
	\subsection*{\normalsize Conflict of interest} The authors declare no competing interests.
	
\end{document}